\documentclass[11pt]{amsart}
\usepackage[T1]{fontenc}
\usepackage[latin1]{inputenc}
\usepackage{a4}
\usepackage{setspace}
\usepackage{xypic}
\usepackage{amssymb}
\usepackage{xcolor}
\usepackage{amsthm}
\usepackage{amscd}
\usepackage[english]{babel}
\usepackage{latexsym}

\newtheorem{thm}{Theorem}[section]
\newtheorem{lem}[thm]{Lemma}
\newtheorem{example}[thm]{Example}
\newtheorem{defn}[thm]{Definition}
\newtheorem*{thm*}{Theorem}
\newtheorem{rem}[thm]{Remark}

\newtheorem{prop}[thm]{Proposition}

\newtheorem{cor}[thm]{Corollary}
\newtheorem{notation}[thm]{Notation}

\newcommand{\Q}{\mathbb Q}

\newcommand{\Z}{\mathbb Z}

\newcommand{\R}{\mathbb R}

\newcommand{\X}{{\mathcal X}}
\newcommand{\Y}{{\mathcal Y}}
\renewcommand{\O}{{\mathcal O}}

\newcommand{\M}{{\mathcal M}}
\newcommand{\Hom}{{\mathrm{Hom}}}
\newcommand{\Spec}{{\mathrm{Spec}}}

\newcommand{\cok}{{\operatorname{coker }}}
\newcommand{\im}{{\operatorname{im }}}
\newcommand{\cone}{\operatorname{cone }}

\newcommand{\rec}{{\mathrm{rec }}}
\newcommand{\Tor}{{\operatorname{Tor}}}

\newcommand{\et}{{\mathrm{et}}}

\title[On Integral Class Field Theory]{On Integral Class field theory for varieties over $p$-adic fields}

\begin{document}

\date{}
\author{Thomas H.\ Geisser}
\address{
	Department of Mathematics, Rikkyo University, Ikebukuro, Tokyo, Japan
}
\email{geisser@rikkyo.ac.jp}
\author{Baptiste Morin}
\address{
	Department of Mathematics, Universit\'e de Bordeaux,
	Bordeaux, France
}
\email{Baptiste.Morin@math.u-bordeaux.fr}
\thanks{The first named author is supported by JSPS Grant-in-Aid (C) 18K03258,
and the second named author by grant ANR-15-CE40-0002}

\begin{abstract}
Let $K$ be a finite extension of the $p$-adic numbers $\Q_p$ with ring of 
integers $\O_K$, $\X$ a regular scheme, proper, flat, and geometrically irreducible
over $\O_K$ of dimension $d$, and $\X_K$ its generic fiber.
We show, under some assumptions on $\X_K$, that
there is a reciprocity isomorphism of locally compact groups
$H_{ar}^{2d-1}(\X_K,\Z(d))\simeq \pi_1^{ab}(\X_K)_{W}$
from the cohomology theory defined in \cite{Geisser-Morin-21} to an
integral model $\pi_1^{ab}(\X_K)_{W}$ of the abelianized geometric fundamental
groups $\pi_1^{ab}(\X_K)^{geo}$. After removing the contribution from the base 
field, the map becomes an isomorphism of finitely generated abelian groups.
The key ingredient is the duality result in \cite{Geisser-Morin-21}.
\end{abstract}

\maketitle


\section{Introduction}
Let $K$ be a finite extension of the $p$-adic numbers $\Q_p$ with ring of 
integers $\O_K$, $\X$ a regular scheme, proper, flat, and geometrically irreducible
over $\O_K$ of dimension $d$, and $\X_K$ its generic fiber.
Classically, the abelianized fundamental group $\pi_1^{ab}(\X_K)$ was 
studied using the reciprocity map
\begin{equation*}
\rho: SK_1(\X_K)\simeq H_\M^{2d-1}(\X_K,\Z(d))\rightarrow\pi^{ab}_1(\X_K).
\end{equation*}
Removing the contribution from the base field we 
let $SK_1(\X_K)^0$ and $\pi_1^{ab}(\X_K)^{geo}$ be the kernels of the norm map 
$SK_1(\X_K)\rightarrow K^{\times}$ and the natural surjection 
$\pi_1^{ab}(\X_K)\to G_K$ to the Galois group of $K$, respectively. 
If $\X_K$ is a curve, then S.\ Saito \cite{Saito85} showed that $\rho$ 
has a divisible kernel, and a cokernel isomorphic to $\widehat{\Z}^r$, and  
the resulting map
\begin{equation}\label{rec0}
SK_1(\X_K)^0\rightarrow\pi_1^{ab}(\X_K)^{geo}
\end{equation}
has divisible kernel, finite image, and a  cokernel isomorphic to $\widehat{\Z}^r$. 
More generally, Yoshida \cite{Yoshida} showed that in arbitrary dimension 
$\pi_1^{ab}(\X_K)^{geo}$ has finite torsion, and its torsion free quotient
is isomorphic to $\hat \Z^r$, where $r$ is the $K$-rank of the special fiber of
the N\'eron model of the Albanese variety of $\X_K$.
If $\X_K$ is a surface, then Sato \cite{Sato05} gave an example showing
that the kernel of 
$\rho$ need not be divisible. On the other hand, Szamuely \cite{Szamuely} showed that
the kernel of $\rho$ is $l$-divisible if $H^2_\et(\X_K,\Q_l)$ vanishes, 
and Jannsen-Saito \cite{Jannsen-Saito}
showed that it is a direct sum of a finite group and a group divisible by
all primes different from $p$. Forr\'e \cite{Forre} generalized this statement
to arbitrary dimension, and Yamazaki \cite{Yamazaki} 
showed that the kernel of $\rho$ is divisible for a product of curves
all but one of which have potentially good reduction.

In this paper, we study
the abelianized fundamental group $\pi_1^{ab}(\X_K)$ using the cohomology
theory $R\Gamma_{ar}(\X_K,\Z(d))$ consisting of locally compact groups
defined in \cite{Geisser-Morin-21}. 
We recall the construction of the cohomology theory, and define an
integral integral model $\pi_1^{ab}(\X_K)_{W}$ of $\pi_1^{ab}(\X_K)$,
similar to the construction in \cite{Geisser-Schmidt-17}.
Our main result is that the
resulting reciprocity map is an isomorphism under some mild assumptions. 
The main tool is the duality theorem of \cite[Cor.\ 5.13]{Geisser-Morin-21},
which is inspired by the Pontryagin duality \cite[Thm.\ 4.9]{Flach-Morin-17}
of the Weil-Arakelov cohomology defined by Flach and the second named author
in \cite{Flach-Morin-17}. 
The advantage of $R\Gamma_{ar}(-,\Z(n))$ is that it is expected to satisfy a 
duality with integral coefficients, whereas \'etale motivic cohomology with 
integral coefficients does not satisfy duality \cite{Geisser16}.


More precisely, let $\X$ be a regular, connected scheme of Krull dimension $d$, 
which is proper and flat over the ring of integers $\O_K$ of 
$K$. The geometric fibers of $\X/\O_K$ are connected, and we assume that they 
are reduced.
We denote by $\X_K$ and $\X_s$ the generic and special fiber of $\X$, respectively.
Following \cite{Geisser06}, we denote by $\Z^c(0)$ Bloch's cycle complex on $\X_s$,
and by $H_{i}^W(\X_s,\Z)$ the $i$th homology group of the complex 
$R\Gamma_W(\X_s,\Z^c(0))[1]$. Let $\Pi^{ab}_1(\X_{s})$ the abelian enlarged 
fundamental group of the closed fiber, a constant pro-group, let
$\Pi^{ab}_1(\X_{s})_W:= \Pi^{ab}_1(\X_s)\times_{G_{\kappa(s)}}W_{\kappa(s)}$,
and define
$$\pi_1^{ab}(\X_K)_W:= \pi^{ab}_1(\X_{K})\times_{\pi^{ab}_1(\X_{s})} 
\Pi_1^{ab}(\X_{s})_W.$$

\begin{thm}[Thm.\ \ref{ICFT}]
Assume that $\X$ has good or strictly semi-stable reduction and that 
$R\Gamma_W(\X_s,\Z^c(0))$ is a perfect complex of abelian groups. 
Then there exists a unique functorial isomorphism of locally compact groups
\begin{equation}\label{intrecintro}
\underline{\rec}:H^{2d-1}_{ar}(\X_{K},\Z(d))\stackrel{\sim}{ \longrightarrow}
\pi_1^{ab}(\X_K)_{W}
\end{equation}
inducing the classical isomorphism
$$H^{2d-1}_{et}(\X_{K},\widehat{\Z}(d))\stackrel{\sim}{ \longrightarrow} 
\pi_1^{ab}(\X_K)$$
after profinite completion. 
\end{thm}

For $\X=\O_K$, the map $\underline{\rec}$ is the isomorphism
$$K^{\times}=H^{1}_{ar}(\X_{K},\Z(1))\stackrel{\sim}{ \longrightarrow} \pi_1^{ab}(\X_K)_{W}=W_K^{ab}$$
of Weil's local class field theory, where $W_K$ is the Weil group of $K$.
The Theorem can be refined as follows.
Setting $\pi_1^{ab}(\X_K)^{geo}_W$ and $H_{ar}^{2d-1}(\X_K,\Z(d))^0$ 
to be the kernel of the map
$\pi_1^{ab}(\X_K)_{W}\rightarrow W_K^{ab}$ and 
$H_{ar}^{2d-1}(\X_K,\Z(d))^0\rightarrow K^{\times}$, respectively, we obtain:

\begin{cor}[Cor.\ \ref{ICFT0}]
The reciprocity map  \eqref{intrecintro} induces an isomorphism of 
{\it finitely generated abelian groups}
$$rec: H_{ar}^{2d-1}(\X_K,\Z(d))^0\stackrel{\sim}{\longrightarrow} 
\pi_1^{ab}(\X_K)^{geo}_W$$
of rank 
\begin{equation*}
r= \mathrm{rank}_{\Z}(H_{ar}^{2d-1}(\X_K,\Z(d))^0)=
\mathrm{rank}_{\Z}(H_1^W(\X_s,\Z))-1.
\end{equation*}
\end{cor}

In the last three sections of this paper, we compare our reciprocity map 
with the classical reciprocity map  \eqref{rec0}. 
For $\X$ of arbitrary dimension, we obtain some results on the kernel 
of $\rho$ in Sections $4$ and $5$. In Section 6 we specialize these
results to obtain the following:

\begin{thm}[Thm.\ \ref{thmreccurve}]
If $\X_K$ is a curve with good or strictly semi-stable reduction, then 
\eqref{rec0} factors as follows:
$$SK_1(\X_K)^0\to  
H_{ar}^3(\X_K,\Z(2))^0 \stackrel{\sim}{\longrightarrow}  
\pi_1^{ab}(\X_K)_W^{geo}\rightarrow  \pi_1^{ab}(\X_K)^{geo}.$$
The left map 
has kernel the maximal divisible subgroup of $SK_1(\X_K)^0$ and its image is 
the torsion subgroup of the finitely generated abelian group 
$H_{ar}^3(\X_K,\Z(2))^0$.
The right map is the inclusion of the finitely generated group 
$\pi_1^{ab}(\X_K)_W^{geo}$ into its profinite completion $\pi_1^{ab}(\X_K)^{geo}$. 
\end{thm}

\section{Review of arithmetic cohomology with Tate twists $0$ and $d$.}

Let $p$ be a prime number, let $K/\Q_p$ be a  finite extension, and let $\bar{K}/K$ 
be an algebraic closure. 
We denote by $K^{un}$ the maximal unramified extension of $K$ inside $\bar{K}$, 
by $\O_K$ and $\O_{K^{un}}$ the rings of integers in $K$ and $K^{un}$ respectively, 
and by $s$ and  $\bar{s}$ the closed points of $\mathrm{Spec}(\O_K)$ and 
$\mathrm{Spec}(\O_{K^{un}})$ respectively.  Let $\X/\O_K$ be a $d$-dimensional, 
connected, regular scheme, proper and flat over $\O_K$. Let $\X_s$ (resp. $\X_K$) 
be its special (resp. generic) fiber. We consider the following diagram.
\[ \xymatrix{
\X_{K^{un}}\ar[r]^{\bar{j}}\ar[d]^{}&\X_{\O_{K^{un}}}\ar[d]& \ar[l]_{\bar{i}}\X_{\bar{s}}\ar[d]\\
\X_{K} \ar[r]^{j} &\X&\ar[l]_{i}\X_{s} 
}
\]
We denote by $\Z(d)$ Bloch's cycle complex, which we consider as a complex of \'etale sheaves. We denote  by $G_{\kappa(s)}\simeq\widehat{\Z}$ and by $W_{\kappa(s)}\simeq\Z$ the Galois group and the Weil group of the finite field $\kappa(s)$. We define 
\begin{eqnarray*}
R\Gamma_W(\X_K,\Z(d))&:=&R\Gamma(W_{\kappa(s)},R\Gamma_{et}(\X_{K^{un}},\Z(d))),\\
R\Gamma_W(\X,\Z(d))&:=&R\Gamma(W_{\kappa(s)},R\Gamma_{et}(\X_{\O_{K^{un}}},\Z(d))),\\
R\Gamma_{W}(\X_{s},Ri^!\Z(d))&:=&R\Gamma(W_{\kappa(s)},R\Gamma_{et}(\X_{\bar{s}},R\bar{i}^!\Z(d)).
\end{eqnarray*}
There is a fiber sequence
\begin{equation}\label{locfibsequenceKun+}
R\Gamma_{et}(\X_{\bar{s}},R\bar{i}^!\Z(d))\rightarrow R\Gamma_{et}(\X_{\O_{K^{un}}},\Z(d))\rightarrow R\Gamma_{et}(\X_{K^{un}},\Z(d)).
\end{equation}
Applying the natural transformation $R\Gamma(G_{\kappa(s)},-)\rightarrow R\Gamma(W_{\kappa(s)},-)$ to (\ref{locfibsequenceKun+}), we obtain
the morphism of fiber sequences in $\mathbf{D}(\mathrm{Ab})$
\[ \xymatrix{
R\Gamma_{et}(\X_{s},Ri^!\Z(d))\ar[d]\ar[r]& R\Gamma_{et}(\X,\Z(d))\ar[d]\ar[r]& R\Gamma_{et}(\X_K,\Z(d))\ar[d]\\
R\Gamma_{W}(\X_{s},Ri^!\Z(d))\ar[r]& R\Gamma_W(\X,\Z(d))\ar[r]& R\Gamma_W(\X_K,\Z(d))
}
\]
The complex $R\Gamma_W(\X,\Z(d))$ is not known to be bounded below. 
However the complex $$R\Gamma_W(\X,\Q/\Z(d))
\simeq R\Gamma_{et}(\X,\Q/\Z(d))$$
is bounded, as can be seen by duality, hence the cohomology groups 
$H^{i}_W(\X,\Z(d))$ are $\Q$-vector spaces for $i<< 0$. In particular, for $a<b<< 0$ the map
$$\tau^{>a}R\Gamma_W(\X,\Z(d))\rightarrow \tau^{>b}R\Gamma_W(\X,\Z(d))$$
induces an equivalence
$$(\tau^{>a}R\Gamma_W(\X,\Z(d)))\widehat{\underline{\otimes}}\widehat{\Z}\stackrel{\sim}{\longrightarrow} (\tau^{>b}R\Gamma_W(\X,\Z(d)))\widehat{\underline{\otimes}}\widehat{\Z}.$$
The same observation applies to $R\Gamma_{et}(\X,\Z(d))$.
Here the functor $$(-)\widehat{\underline{\otimes}}\widehat{\Z}:\mathbf{D}^b(\mathrm{Ab})\rightarrow \mathbf{D}^b(\mathrm{LCA})$$ is defined in \cite[Section 2.3]{Geisser-Morin-21}, where $\mathbf{D}^b(\mathrm{Ab})$ (resp. $\mathbf{D}^b(\mathrm{LCA})$) denotes the derived $\infty$-category of bounded complexes of abelian groups (respectively of bounded complexes of locally compact abelian groups),  see \cite{Hoffmann-Spitzweck-07}  and \cite[Section 2.2]{Geisser-Morin-21}.

\begin{defn}
Let $a\ll 0$. We define $$R\Gamma_{ar}(\X,\Z(d)):= (\tau^{>a}R\Gamma_{W}(\X,\Z(d)))\widehat{\underline{\otimes}}\widehat{\Z}.$$
If $R\Gamma_W(\X_s,\Z^c(0))$ is cohomologically bounded, we define $R\Gamma_{ar}(\X_K,\Z(d))$ as the cofiber of the composite map
$$R\Gamma_W(\X_s,Ri^!\Z(d))\rightarrow \tau^{>a}R\Gamma_{W}(\X,\Z(d)) \rightarrow R\Gamma_{ar}(\X,\Z(d))$$
in $\mathbf{D}^b(\mathrm{LCA})$. Similarly, we define
$$R\widehat{\Gamma}_{et}(\X,\Z(d)):=(\tau^{>a}R\Gamma_{et}(\X,\Z(d)))\widehat{\underline{\otimes}}\widehat{\Z}.$$
and  $R\widehat{\Gamma}_{et}(\X_K,\Z(d))$ as the cofiber of the composite map
$$R\Gamma_{et}(\X_s,Ri^!\Z(d))\rightarrow \tau^{>a}R\Gamma_{et}(\X,\Z(d))\rightarrow R\widehat{\Gamma}_{et}(\X,\Z(d))$$
in $\mathbf{D}^b(\mathrm{LCA})$. 
\end{defn}
\begin{rem}On the connected, $d$-dimensional and regular scheme $\X$, we have  $\Z(d)^{\X}=\Z^c(0)^{\X}[-2d]$ by definition. By \cite[Cor.\  7.2]{Geisser10}, we have $Ri^!\Z^c(0)^{\X}=\Z^c(0)^{\X_s}$, hence $Ri^!\Z(d)^{\X}=\Z^c(0)^{\X_s}[-2d]$. Moreover, the following assertions are equivalent:
\begin{itemize}
\item $R\Gamma_W(\X_s,\Z^c(0))$ is cohomologically bounded;
\item $R\Gamma_W(\X_s,\Q^c(0))$ is cohomologically bounded;
\item $R\Gamma_{et}(\X_s,\Z^c(0))$ is cohomologically bounded;
\item $R\Gamma_{et}(\X_s,\Q^c(0))$ is cohomologically bounded.
\end{itemize}
Indeed, both $H^{i}_W(\X_s,\Z^c(0))$ and $H^{i}_{et}(\X_s,\Z^c(0))$ are $\Q$-vector spaces for $i<<0$. Moreover all these complexes are bounded above. Hence the result follows from the direct sum decomposition 
$$R\Gamma_W(\X_s,\Q^c(0))\simeq R\Gamma_{et}(\X_s,\Q^c(0))\oplus R\Gamma_{et}(\X_s,\Q^c(0))[-1].$$
\end{rem}

If $R\Gamma_W(\X_s,\Z^c(0))$ is cohomologically bounded, then  we have a commutative diagram in $\mathbf{D}^b(\mathrm{LCA})$:
\[ \xymatrix{
R\Gamma_{et}(\X_{s},Ri^!\Z(d))\ar[d]\ar[r]& R\Gamma_{et}(\X,\Z(d))\ar[d]\ar[r]& R\Gamma_{et}(\X_K,\Z(d))\ar[d]\\
R\Gamma_{et}(\X_{s},Ri^!\Z(d))\ar[d]\ar[r]& R\widehat{\Gamma}_{et}(\X,\Z(d))\ar[d]\ar[r]& R\widehat{\Gamma}_{et}(\X_K,\Z(d))\ar[d]\\
R\Gamma_{W}(\X_{s},Ri^!\Z(d))\ar[r]& R\Gamma_{ar}(\X,\Z(d))\ar[r]& R\Gamma_{ar}(\X_K,\Z(d))
}
\]
where the rows are fiber sequences, and the top row consists of discrete objects in $\mathbf{D}^b(\mathrm{LCA})$ in the sense of \cite[Def.\ 2.7]{Geisser-Morin-21}.

\begin{prop}\label{propeqarethat}
The map $$R\widehat{\Gamma}_{et}(\X,\Z(d))\rightarrow R\Gamma_{ar}(\X,\Z(d))$$
is an equivalence. If $R\Gamma_W(\X_s,\Z^c(0))$ is cohomologically bounded, 
then we have a canonical cofiber sequence
\begin{equation}\label{descent-triangle-top-redefined-n=d}
R\Gamma_{et}(\X_s,Ri^{!}\Q(d))[-1]\rightarrow 
R\widehat{\Gamma}_{et}(\X_K,\Z(d))\rightarrow R\Gamma_{ar}(\X_K,\Z(d)).
\end{equation}
\end{prop}
\begin{proof}
In view of the equivalence $$(\tau^{>a}R\widehat{\Gamma}_{et}(\X,\Z(d)))\otimes^L\Z/m\Z \stackrel{\sim}{\rightarrow} (\tau^{>a}R\Gamma_{W}(\X,\Z(d)))\otimes^L\Z/m\Z$$
for any $m$, the first assertion follows from \cite[Remark 2.12]{Geisser-Morin-21}.

If $R\Gamma_W(\X_s,\Z^c(0))$ is cohomologically bounded, then we have a commutative square 
\[ \xymatrix{
 R\Gamma_{et}(\X_{s},Ri^{!}\Z(d)) \ar[r]\ar[d]& R\Gamma_{W}(\X_{s},Ri^{!}\Z(d))\ar[d]\\
R\widehat{\Gamma}_{et}(\X,\Z(d))\ar[r]^{\sim}& R\Gamma_{ar}(\X,\Z(d))
}
\]
in $\mathbf{D}^b(\mathrm{LCA})$ and a fiber sequence
$$R\Gamma(\X_{s},Ri^{!}\Q(d)))[-2]\rightarrow R\Gamma_{et}(\X_{s},Ri^{!}\Z(d)) \rightarrow R\Gamma_{W}(\X_{s},Ri^{!}\Z(d))$$
of discrete objects in $\mathbf{D}^b(\mathrm{LCA})$ in the sense of \cite[Def.\ 2.7]{Geisser-Morin-21}. We obtain the following diagram
\[ \xymatrix{
R\Gamma_{et}(\X_{s},Ri^{!}\Q(d)))[-2] \ar[r]\ar[d]^{}& R\Gamma_{et}(\X_{s},Ri^{!}\Z(d)) \ar[r]\ar[d]& R\Gamma_{W}(\X_{s},Ri^{!}\Z(d))\ar[d]\\
0\ar[r]\ar[d]&R\widehat{\Gamma}_{et}(\X,\Z(d))\ar[d]\ar[r]^{\sim}& R\Gamma_{ar}(\X,\Z(d))\ar[d]
\\
R\Gamma_{et}(\X_{s},Ri^{!}\Q(d)))[-1]\ar[r]&R\widehat{\Gamma}_{et}(\X_{K},\Z(d))\ar[r]& R\Gamma_{ar}(\X_{K},\Z(d))
}
\]
with exact rows and columns. The lower horizontal cofiber sequence is (\ref{descent-triangle-top-redefined-n=d}).
\end{proof}

\begin{notation}
We set
$$R\Gamma_{et}(\X,\widehat{\Z}(d)):= R\Gamma_{et}(\X,\Z(d)) \widehat{\otimes} \widehat{\Z}:= R\mathrm{lim} \left(R\Gamma_{et}(\X,\Z(d))\otimes^L \Z/m\Z\right)$$
$$R\Gamma_{et}(\X_K,\widehat{\Z}(d)):= R\Gamma_{et}(\X_K,\Z(d)) \widehat{\otimes} \widehat{\Z}:= R\mathrm{lim} \left(R\Gamma_{et}(\X_K,\Z(d))\otimes^L \Z/m\Z\right)$$
in $\mathbf{D}^b(\mathrm{Ab})$.
\end{notation}

Recall from \cite[Prop.\  2.6]{Geisser-Morin-21} that we have an adjunction
$$\iota:\mathbf{D}^b(\mathrm{Ab})\rightleftarrows \mathbf{D}^b(\mathrm{LCA}):\mathrm{disc}$$
We sometimes denote $(-)^{\delta}:=\mathrm{disc}(-)$ for brevity and suppress any mention of the functor $\iota$.

\begin{prop}\label{ar-et-hat-et-complete}
We have equivalences
\begin{equation}\label{forgettopmap}
\mathrm{disc}(R\widehat{\Gamma}_{et}(\X,\Z(d)))\simeq 
\mathrm{disc}(R\Gamma_{ar}(\X,\Z(d)))\simeq R\Gamma_{et}(\X,\widehat{\Z}(d))
\end{equation}
in $\mathbf{D}^b(\mathrm{Ab})$. If $R\Gamma_W(\X_s,\Z^c(0))$
is cohomologically bounded, then we have a canonical map
\begin{equation}\label{forgettopmap2}
\mathrm{disc}(R\Gamma_{ar}(\X_K,\Z(d)))\rightarrow R\Gamma_{et}(\X_K,\widehat{\Z}(d))
\end{equation}
in $\mathbf{D}^b(\mathrm{Ab})$, which induces an equivalence after applying $(-)\widehat{\otimes} \widehat{\Z}$.
\end{prop}
\begin{proof}
The equivalences of (\ref{forgettopmap}) are given by Proposition \ref{propeqarethat} and \cite[Remark 4.12]{Geisser-Morin-21} respectively. If $R\Gamma_W(\X_s,\Z^c(0))$ is cohomologically bounded, then we have a morphism of cofiber sequences
\[ \xymatrix{
R\Gamma_{W}(\X_s,Ri^{!}\Z(d))\widehat{\otimes} \widehat{\Z} \ar[r]\ar[d]^{\simeq}&R\Gamma_{ar}(\X,\Z(d))^{\delta}\widehat{\otimes} \widehat{\Z}\ar[d]^{\simeq}\ar[r]&R\Gamma_{ar}(\X_K,\Z(d))^{\delta}\widehat{\otimes} \widehat{\Z}    \ar[d]\\
R\Gamma_{et}(\X_s,Ri^{!}\Z(d))\widehat{\otimes} \widehat{\Z} \ar[r]& R\Gamma_{et}(\X,\widehat{\Z}(d))\ar[r]& R\Gamma_{et}(\X_K,\widehat{\Z}(d))  
}
\]
where the left and middle vertical maps are equivalences. Hence the right vertical map is an equivalence as well. 

\end{proof}
\begin{rem}\label{remfunctoriality}
Using the push-forward map of \cite[Cor.\  7.2]{Geisser10}, it is easy to check that the complexes $R\Gamma_{ar}(\X,\Z(d_{\X}))$ and $R\Gamma_{ar}(\X_K,\Z(d_{\X}))$ are covariantly functorial for proper maps $\X\rightarrow \mathcal{Y}$, and that the equivalences (\ref{forgettopmap})  and the map (\ref{forgettopmap2}) are functorial.
\end{rem}
\begin{example}
We have $R\Gamma_{et}(\mathrm{Spec}(\O_K),\Z(1))\simeq R\Gamma(G_{\kappa(s)}, \O^{\times}_{K^{un}})\simeq \O^{\times}_{K}[-1]$ in $\mathbf{D}^b(\mathrm{Ab})$. We obtain
\begin{eqnarray*}
R\Gamma_{ar}(\mathrm{Spec}(\O_K),\Z(1))&\simeq&R\widehat{\Gamma}_{et}(\mathrm{Spec}(\O_K),\Z(1))\\
&\simeq&  (\tau^{>a}R\Gamma_{et}(\mathrm{Spec}(\O_K),\Z(1))\widehat{\underline{\otimes}}\widehat{\Z}\\
&\simeq& \left(\O^{\times}_K[-1] \right)\widehat{\underline{\otimes}}\widehat{\Z}\\
&\simeq& \O^{\times}_K[-1]
\end{eqnarray*}
 where $\O^{\times}_K$ is endowed with its profinite topology. Here we use Proposition \ref{propeqarethat} and 
 the fact that $\O^{\times}_K$ is derived profinite complete. Moreover, we have $Ri^!\Z(1)\simeq \Z[-2]$, $H_W^{i}(s,\Z)=\Z$ for $i=0,1$, and $H_W^{i}(s,\Z)=0$ for $i\neq 0,1$. On the other hand, we have $H_{et}^{i}(s,\Z)=\Z,0,\Q/\Z$ for $i=0,1,2$, and $H_{et}^{i}(s,\Z)=0$ for $i\neq 0,2$. We obtain
\begin{eqnarray*}
H^{i}_{ar}(\mathrm{Spec}(K),\Z(1))&\simeq& K^{\times},\Z\textrm{ for }i=1,2,\\
\widehat{H}^{i}_{et}(\mathrm{Spec}(K),\Z(1))&\simeq& K^{\times},0,\Q/\Z\textrm{ for }i=1,2,3
\end{eqnarray*}
where $K^{\times}$ is endowed with its natural topology, $H^{i}_{ar}(\mathrm{Spec}(K),\Z(1))=0$ for $i\neq 1,2$, and $\widehat{H}^{i}_{et}(\mathrm{Spec}(K),\Z(1))=0$ for $i\neq 1,3$.
\end{example}

\subsection{Weight $0$}
Recall that for a finite field $k$ with algebraic closure $\bar k$ and 
$W_k$ be its Weil-group, the Weil-\'etale and $Wh$-cohomology of a proper 
scheme $Y$ over $k$ of the constant sheaf $\Z$ are defined to be
$$R\Gamma_{W}(Y,\Z):= R\Gamma(W_k, R\Gamma_{et}(Y\times_k\bar k,\Z)),$$
$$R\Gamma_{Wh}(Y,\Z):= R\Gamma(W_k, R\Gamma_{eh}(Y\times_k\bar k,\Z)).$$
By \cite[Prop.\ 3.2, Prop.\ 3.5]{Geisser-Morin-21}, 
$R\Gamma_{W}(Y,\Z)$ is a perfect complex of abelian groups if $Y^{red}$
is a strict normal crossing scheme, and $R\Gamma_{Wh}(Y,\Z)$ is a 
perfect complex of abelian groups under resolution of singularities.

\begin{notation}
For the rest of the paper we will write $R\Gamma_{ar}(Y,\Z)$ in both of
the above cases. Since $R\Gamma_{W}(Y,\Z)\simeq R\Gamma_{Wh}(Y,\Z)$
if $Y^{red}$ is a strict normal crossing scheme under resolution of singularities,
there is no conflict in the intersection of both cases. 
\end{notation}

If $R\Gamma_{ar}(Y,\Z)$ is a perfect complex of abelian groups, it 
belongs to $\mathbf{D}^b(\mathrm{FLCA})$, 
where $\mathbf{D}^b(\mathrm{FLCA})\subseteq \mathbf{D}^b(\mathrm{LCA})$ is 
the full stable subcategory consisting of 
bounded complexes of locally compact abelian groups of finite ranks in the sense 
of \cite{Hoffmann-Spitzweck-07}. Hence we can define 
$$ 
R\Gamma_{ar}(\X_s,\R/\Z):=
R\Gamma_{ar}(\X_s,\Z)\underline{\otimes}^L\R/\Z \in \mathbf{D}^b(\mathrm{FLCA}).$$

Recall from \cite[Section 4.3, Notation 4.7]{Geisser-Morin-21} the definition of 
the complexes $R\Gamma_{ar}(\X,\Z)$ and $R\Gamma_{ar}(\X_K,\Z)$:
If $\X_s^{\mathrm{red}}$ is a simple normal crossing scheme in the sense of 
\cite[Def.\ 3.1]{Geisser-Morin-21} (in particular  when $\X$ has strictly 
semi-stable reduction in the sense of \cite[Def.\ 4.15]{Geisser-Morin-21}),
or if we assume resolution of singularities, we define
$R\Gamma_{ar}(\X,\Z)=R\Gamma_{ar}(\X_s,\Z)
\in\mathbf{D}^b(\mathrm{LCA})$, and 
$R\Gamma_{ar}(\X_K,\Z)\in\mathbf{D}^b(\mathrm{LCA})$ by 
the fiber sequences
\begin{equation*}
R\Gamma_{W}(\X_s,Ri^{!}\Z)\rightarrow R\Gamma_{ar}(\X,\Z)\rightarrow R\Gamma_{ar}(\X_K,\Z)
\end{equation*}
in $\mathbf{D}^b(\mathrm{LCA})$. 

The following result is \cite[Prop.\  4.13]{Geisser-Morin-21}, where
$\mathbf{D}^b(\mathrm{FLCA})\subseteq \mathbf{D}^b(\mathrm{LCA})$ is 
the full stable subcategory consisting of 
bounded complexes of locally compact abelian groups of finite ranks in the sense 
of \cite{Hoffmann-Spitzweck-07}.

\begin{prop} \label{prop-finiteranks}
1) Assume resolution of singularities for schemes over $\kappa(s)$ of dimension 
$\leq d-1$, or assume that $\X_s^{\mathrm{red}}$ is a simple normal crossing scheme. 
Then $R\Gamma_{ar}(\X,\Z)$ and $R\Gamma_{ar}(\X_K,\Z)$ belong to 
$\mathbf{D}^b(\mathrm{FLCA})$.

2) Assume that $R\Gamma_W(\X_s,\Z^c(0))$ is a perfect complex of abelian groups.
Then $R\Gamma_{ar}(\X,\Z(d))$ and $R\Gamma_{ar}(\X_K,\Z(d))$ belong to 
$\mathbf{D}^b(\mathrm{FLCA})$.
\end{prop}

If $R\Gamma_{ar}(\X_K,\Z)$ belong to $\mathbf{D}^b(\mathrm{FLCA})$, 
then we define \cite{Geisser-Morin-21}
\begin{eqnarray*}
R\Gamma_{ar}(\X_K,\R)&:=&R\Gamma_{ar}(\X_K,\Z)\underline{\otimes}^L\R \in \mathbf{D}^b(\mathrm{FLCA})\\
R\Gamma_{ar}(\X_K,\R/\Z)&:=&R\Gamma_{ar}(\X_K,\Z)\underline{\otimes}^L\R/\Z \in \mathbf{D}^b(\mathrm{FLCA})
\end{eqnarray*}
where $\underline{\otimes}^L$ denotes the derived tensor product in $\mathbf{D}^b(\mathrm{FLCA})$. We have a fiber sequence
$$R\Gamma_{ar}(\X_K,\Z)\rightarrow R\Gamma_{ar}(\X_K,\R)\rightarrow R\Gamma_{ar}(\X_K,\R/\Z).$$
The cohomology groups $H^{i}_{ar}(\X_K,A)\in\mathcal{LH}(\mathrm{FLCA})$ for $A=\Z,\R,\R/\Z$ are defined using the $t$-structure on the stable $\infty$-category $\mathbf{D}^b(\mathrm{FLCA})$, whose heart is the abelian category $\mathcal{LH}(\mathrm{FLCA})$. A fiber sequence in $\mathbf{D}^b(\mathrm{FLCA})$ therefore yields a long exact sequence in $\mathcal{LH}(\mathrm{FLCA})$.

\section{Class Field Theory}
From now on, $\X/\O_K$ denotes a flat, proper, separable $\O_K$-scheme with 
geometrically connected fibers. We assume that $\X$ is regular of Krull dimension $d$.  

\subsection{The abelian Weil-\'etale fundamental group}

We denote by $\pi^{ab}_1(S)$ the abelianization of the classical \'etale 
profinite fundamental group of a scheme $S$. We have morphisms
$$\pi^{ab}_1(\X_{K})\twoheadrightarrow  \pi^{ab}_1(\X)
\stackrel{\sim}{\leftarrow} \pi^{ab}_1(\X_s)$$
where the map on the left map is surjective because $\X$ is normal 
\cite[V, Prop.\ 8.2]{SGA1} and the right map is an isomorphism by 
\cite[X, Thm 2.1]{SGA1}.
We denote by $\Pi^{ab}_1(\X_{s})$ the abelian enlarged fundamental group of the 
closed fiber. The pro-group $\Pi_1^{ab}(\X_{s})$ is isomorphic to a 
constant abelian group.

\begin{defn} 
We define
$$\Pi^{ab}_1(\X_{s})_W:= \Pi^{ab}_1(\X_s)\times_{G_{\kappa(s)}}W_{\kappa(s)}$$
and
$$\pi_1^{ab}(\X_K)_W:= \pi^{ab}_1(\X_{K})\times_{\pi^{ab}_1(\X_{s})} \Pi_1^{ab}(\X_{s})_W.$$
\end{defn}

By \cite[Theorem 2.5]{Geisser-Schmidt-17}, $\Pi^{ab}_1(\X_{s})_W$ is a finitely generated abelian group, 
hence $\pi_1^{ab}(\X_K)_W$ is an extension of finitely generated abelian group by a profinite abelian group. 
In particular, $\pi_1^{ab}(\X_K)_W$ is locally compact. 
Moreover, $\Pi^{ab}_1(\X_{s})_W$ is a  dense subgroup of $\pi^{ab}_1(\X_{s})$, hence 
$\pi_1^{ab}(\X_K)_W$ is a dense subgroup of $\pi_1^{ab}(\X_K)$. We have a morphism of extensions
\[ \xymatrix{
0\ar[r]^{}&I^{ab}(\X) \ar[r]\ar[d]^{=}&\pi_1^{ab}(\X_K)_W \ar[r]\ar[d]^{}&\Pi^{ab}_1(\X_{s})_W\ar[d]\ar[r]^{}&0\\
0\ar[r]^{}&I^{ab}(\X) \ar[r]& \pi_1^{ab}(\X_K) \ar[r] & \pi^{ab}_1(\X_{s}) \ar[r]^{}&0
}
\]
where $I^{ab}(\X)$ is defined by the lower exact sequence, and the middle and right vertical arrows are continuous, injective with dense image.  

\begin{example}
If $\X=\mathrm{Spec}(\O_{K})$ then
$$\pi_1^{ab}(\X_K)_W\simeq G^{ab}_K\times_{G_{\kappa}}W_{\kappa}\simeq W^{ab}_K.$$
\end{example}

\begin{rem}
The group $\Pi^{ab}_1(\X_{s})$ is the abelianization of the fundamental group of 
the small \'etale topos  $\X_{s,et}$. 
It follows from \cite[Theorem 3.1, Corollary 3]{Flach-Morin-12} that  
$\Pi^{ab}_1(\X_{s})_W$ is the abelianization of the fundamental group of the 
Weil-\'etale topos $\X_{s,W}$. A geometrical interpretation of the fundamental group 
$\pi_1^{ab}(\X_K)_W$
is unknown yet. 
\end{rem}

\begin{notation}
If $A,B$ are locally compact abelian groups, then we denote by $\underline{\Hom}(A,B)$ the abelian group of continuous morphisms from $A$ to $B$, endowed with the compact-open topology. For any  locally compact abelian group $A$, we denote by
$A^D:=\underline{\Hom}(A,\R/\Z)$
the Pontryagin dual of $A$.
\end{notation}

\begin{lem} 1) If $\mathcal{Y}$ is a normal scheme of finite type over $\O_K$, then
$R\Gamma_{et}(\Y,\Z)$ is in $\mathbf{D}^b(\mathrm{FLCA})$ and 
$H_{et}^{i}(\mathcal{Y}, \Q/\Z) \simeq H_{et}^{i}(\mathcal{Y}, \R/\Z)$
for all $i\geq 1$. 

2) If $T$ is a scheme of finite type over a finite field such that
$R\Gamma_{Wh}(T,\Z)$ belongs to $\mathbf{D}^b(\mathrm{FLCA})$, then
$$H_{Wh}^i(T, \Q/\Z)\hookrightarrow H_{Wh}^i(T, \R/\Z)$$
is injective for any $i$. 
\end{lem}

\begin{proof}
1) Since $\mathcal{Y}$ is normal, then $H_{et}^0(\mathcal{Y},\Z)=\Z$, and 
for $i\geq 1$ we have $H_{et}^i(\mathcal{Y},\Q)=0$. Thus 
$H_{et}^i(\mathcal{Y},\Z)$ is the image of the $H_{et}^{i-1}(\mathcal{Y},\Q/\Z)$,
a group of cofinite type \cite[Thm.\ 1.1]{Finitude}. Hence  
$R\Gamma_{et}(\Y,\Z)$ is in $\mathbf{D}^b(\mathrm{FLCA})$ and 
$H^{i}_{et}(\mathcal{Y},\R)=0$ for $i\geq 1$. We can conclude with 
the map of long exact coefficient sequences
$$\begin{CD}
\cdots H_{et}^i(\mathcal{Y},\Z)@>>> H_{et}^i(\mathcal{Y},\Q)@>>> 
H_{et}^i(\mathcal{Y},\Q/\Z)@>>>H_{et}^{i+1}(\mathcal{Y},\Z)\cdots\\
@|@VVV@VVV@|\\
\cdots H_{et}^i(\mathcal{Y},\Z)@>>> H_{et}^i(\mathcal{Y},\R)@>>> 
H_{et}^i(\mathcal{Y}, \R/\Z)@>>>H_{et}^{i+1}(\mathcal{Y},\Z)\cdots
\end{CD}$$

2) Comparing the analog coefficient sequences, we see that the kernel
of $H_{Wh}^i(T,\Q)\to H_{Wh}^i(T,\R)$ surjects
onto the kernel in question, but this map is injective because the groups
are finitely generated.
\end{proof}

\begin{prop}\label{sameH1}
Under resolution of singularities for schemes of dimension at most $\dim T$,
we have
$$H_{W}^1(T,\R/\Z)\simeq H_{Wh}^1(T,\R/\Z)$$
for any $T$ of finite type over a finite field.
\end{prop}

\begin{proof}
Both sides do not change if we replace $T$ by $T^{red}$, so we
assume that $\X_s$ is reduced.
We first claim that the result holds for normal $T$. We first show that
for an integral scheme $S$ over a field, one has $H^1_{eh}(S,\Z)=0$. 
Consider the Hochschild-Serre spectral sequence 
$H^s_{eh}(S,Rj_*\Z)\Rightarrow H^{s+t}_{eh}(\eta,\Z)$
associated to the inclusion of the generic point $j:\eta\to S$.
Since $j_*\Z\cong \Z$,
we obtain an inclusion $H^1_{eh}(S,\Z)\hookrightarrow H^{1}_{eh}(\eta,\Z)$,
but the latter group vanishes because every abstract blow-up cover of a field
splits, so that this group agrees with the Galois cohomology of $\Z$, which
is zero. Now we obtain
$$\Z^{\pi_0(T)}\simeq H^0_{et}(\bar T,\Z)_W\simeq H_{W}^1(T,\Z)$$
and similarly $H_{Wh}^1(T,\Z)\simeq \Z^{\pi_0(T)}$. Thus there
is short exact sequences
$$\begin{CD}
0@>>> (\Q/\Z)^{\pi_0(\Y)}@>>> 
H_{W}^1(T,\Q/\Z)@>>>\Tor H_{W}^{2}(T,\Z)@>>>0\\
@.@|@VVV@VVV\\
0@>>>(\Q/\Z)^{\pi_0(T)}@>>> 
H_{Wh}^1(T,\Q/\Z)@>>>\Tor H_{Wh}^{2}(T,\Z)@>>>0.
\end{CD}$$
The proper base-change theorem implies that middle map, hence that the right hand
map is an isomorphism. 
As $H_{W}^{2}(T,\Z)$ is finitely generated discrete, the kernel of
$H_{W}^{2}(T,\Z)\to H_{W}^{2}(T,\R)\simeq H_{W}^{2}(T,\Z)\underline{\otimes} \R$ is 
the torsion subgroup of $H_{W}^{2}(T,\Z)$ as an abelian group, 
and similarly for $H_{W}^1(T,\Z)$.
Thus we obtain a map of short exact sequences 
$$\begin{CD}
0@>>> (\R/\Z)^{\pi_0(T)}@>>> 
H_{W}^1(T,\R/\Z)@>>>\Tor H_{W}^{2}(T,\Z)@>>>0\\
@.@|@VVV@|\\
0@>>>(\R/\Z)^{\pi_0(T)}@>>> 
H_{Wh}^1(T,\R/\Z)@>>>\Tor H_{Wh}^{2}(T,\Z)@>>>0
\end{CD}$$
and conclude that the middle map is an isomorphism.

In the general case let $T'\to T$ be the normalization, 
$Y$ the closed subscheme with the reduced structure
of smaller dimension where $T'\to T$ is not an isomorphism, 
and $Y'=Y\times_{T}T'$. 
Consider the abstract blow-up sequence of \cite[Lemma 3.4]{Geisser-Morin-21},
$$\begin{CD}
H_{W}^i(T,\R/\Z)@>>> H_{W}^i(T',\R/\Z)\oplus H_{W}^i(Y,\R/\Z)
@>>>H_{W}^i(Y',\R/\Z)\\
@VVV@VVV@VVV\\
H_{Wh}^i(T,\R/\Z)@>>> H_{Wh}^i(T',\R/\Z)\oplus H_{Wh}^i(Y,\R/\Z)
@>>> H_{Wh}^i(Y',\R/\Z)
\end{CD}$$
and proceed by induction on $\dim T$, using that
$H_{W}^0(S,\R/\Z)\simeq H_{Wh}^0(S,\R/\Z)\simeq (\R/\Z)^{\pi_0(S)}$
for any scheme $S$ of finite type over a finite field.
\end{proof}

\begin{prop}\label{pi1corepresent}
Assume resolution of singularities for schemes over $\kappa(s)$ of dimension 
$\leq d-1$, or assume that $\X$ has good or  strictly semi-stable reduction.
Then we have a canonical isomorphism of locally compact groups
$$H_{ar}^1(\X_K,\R/\Z)\simeq \underline{\Hom}(\pi^{ab}_1(\X_K)_W, \R/\Z).$$
 \end{prop}

\begin{proof} 
Assume resolution of singularities for schemes over $\kappa(s)$ of dimension 
$\leq d-1$. By the remark after Proposition \ref{prop-finiteranks}, 
$R\Gamma_{ar}(\X_K,\R/\Z)$ 
is well defined. One has a morphism of fiber sequences 
\[ \xymatrix{
R\Gamma_{et}(\X_s,Ri^{!}\Z) \ar[r]\ar[d]^{\simeq}&R\Gamma_{et}(\X,\Z) \ar[r]\ar[d]^{}&R\Gamma_{et}(\X_K,\Z)\ar[d]\\
R\Gamma_{W}(\X_s,Ri^{!}\Z) \ar[r]& R\Gamma_{ar}(\X_s,\Z) \ar[r] & R\Gamma_{ar}(\X_K,\Z) 
}
\]
where the left vertical map is an equivalence since $Ri^{!}\Z$ is torsion. 
It follows that the square on the right is a push-out square of (bounded) 
discrete complexes. The complexes of the top row are of finite ranks, since 
their cohomology is torsion in degrees $>0$ and finite modulo $l$ for all $l$. 
The complex  of abelian groups $R\Gamma_{ar}(\X_s,\Z)$ is perfect by 
\cite[Prop.\ 3.2, 3.5]{Geisser-Morin-21}. 
Hence the right square above is a push-out in $\mathbf{D}^b(\mathrm{FLCA})$. 
Applying $(-)\underline{\otimes}^L\R/\Z$, we obtain the push-out square
\[ \xymatrix{
R\Gamma_{et}(\X,\R/\Z) \ar[r]\ar[d]^{}&R\Gamma_{et}(\X_K,\R/\Z)\ar[d]\\
R\Gamma_{ar}(\X_s,\R/\Z) \ar[r] & R\Gamma_{ar}(\X_K,\R/\Z) 
}
\]
hence an exact sequence in $\mathcal{LH}(\mathrm{FLCA})$
$$\cdots \rightarrow H_{et}^1(\X, \R/\Z) \stackrel{t}{\rightarrow} 
H_{et}^1(\X_K, \R/\Z)\oplus  H_{ar}^1(\X_s, \R/\Z)\rightarrow  H_{ar}^1(\X_K, \R/\Z)$$
$$\rightarrow H_{et}^2(\X, \R/\Z)\rightarrow  
H_{et}^2(\X_K, \R/\Z)\oplus  H_{ar}^2(\X_s, \R/\Z).$$
We have $H_{et}^i(\X, \Q)=0$ for $i\geq 1$ since $\X$ is normal. It follows that 
$H_{et}^i(\X, \R):=H^{i}(R\Gamma_{et}(\X,\Z)\underline{\otimes}^L\R)=0$
for $i\geq 1$, hence 
$$H_{et}^{i}(\X, \R/\Z)\simeq H_{et}^{i+1}(\X, \Z)
\simeq H_{et}^{i}(\X, \Q/\Z)$$
for any $i\geq 1$. 
Moreover the morphism
$$H_{ar}^i(\X_s, \Q/\Z)\hookrightarrow H_{ar}^i(\X_s, \R/\Z)$$
is injective since the groups $H_{ar}^i(\X_s, \Z)$ are finitely generated. 
In view of the isomorphisms
$$H_{et}^i(\X, \Z/m\Z)\stackrel{\sim}{\rightarrow} 
H_{et}^i(\X_s, \Z/m\Z) \stackrel{\sim}{\rightarrow} H_{ar}^i(\X_s, \Z/m\Z)$$
given by proper base change, we see that the  morphism
$$H_{et}^i(\X, \R/\Z)\simeq H_{et}^i(\X, \Q/\Z) \stackrel{\sim}{\rightarrow} 
H_{ar}^i(\X_s, \Q/\Z)\hookrightarrow H_{ar}^i(\X_s, \R/\Z)$$
is injective for $i\geq 1$.
We obtain an exact sequence in $\mathcal{LH}(\mathrm{LCA})$
\begin{equation}\label{coh}
0\rightarrow H_{et}^1(\X, \R/\Z)\stackrel{t}{\rightarrow}  
H_{et}^1(\X_K, \R/\Z)\oplus  H_{ar}^1(\X_s, \R/\Z)\rightarrow  
H_{ar}^1(\X_K, \R/\Z)\rightarrow 0.
\end{equation}
On the other hand, by definition of $\pi^{ab}_1(\X_K)_{W}$, 
we have a strictly exact sequence 
\begin{equation*}
0\rightarrow \pi^{ab}_1(\X_K)_{W}\stackrel{\iota}{\rightarrow} 
\pi^{ab}_1(\X_K)\oplus \Pi^{ab}_1(\X_s)_{W}\stackrel{s}{\rightarrow} 
\pi^{ab}_1(\X)\rightarrow 0
\end{equation*}
since $\iota$ is a closed embedding by definition, and since $s$ is an open surjection (as $\pi^{ab}_1(\X_K)\rightarrow\pi^{ab}_1(\X)$ is an open surjection), 
see \cite[Section 1]{Hoffmann-Spitzweck-07}. 
The dual of a strictly exact sequence in $\mathrm{LCA}$ is also strictly exact, 
and the fully faithful functor 
$\mathrm{LCA}\hookrightarrow \mathcal{LH}(\mathrm{LCA})$ 
sends a strictly exact sequence in $\mathrm{LCA}$ to an exact sequence in the 
abelian category $\mathcal{LH}(\mathrm{LCA})$, 
see \cite[Cor.\  1.2.28]{Schneiders99}.
Hence we have an exact sequence
$$
0\rightarrow  \underline{\Hom}(\pi^{ab}_1(\X), \R/\Z)
\stackrel{s^D}{\rightarrow}  
\underline{\Hom}(\pi^{ab}_1(\X_K), \R/\Z)\oplus  
\underline{\Hom}(\Pi^{ab}_1(\X_s)_{W}, \R/\Z)$$
$$\rightarrow  \underline{\Hom}(\pi^{ab}_1(\X_K)_{W}, \R/\Z) \rightarrow 0.$$
in $\mathcal{LH}(\mathrm{LCA})$.  

By \cite[Prop.\  2.4]{Geisser-Schmidt-17} and Proposition \ref{sameH1}, 
we have canonical isomorphisms of discrete abelian groups
$$\Hom(\Pi^{ab}_1(\X_s)_{W},A)\stackrel{\sim}{\rightarrow} 
H_{W}^1(\X_s,A)  \stackrel{\sim}{\rightarrow} H_{Wh}^1(\X_s,A)$$
for $A=\Z, \R, \R/\Z$. Moreover, the map $\Hom(\Pi^{ab}_1(\X_s)_{W},\R)\rightarrow H_{W}^1(\X_s,\R)$ 
is an $\R$-linear map between finite dimensional $\R$-vector spaces, hence bicontinuous with respect 
with the Euclidean topology. This implies that 
$\underline{\Hom}(\Pi^{ab}_1(\X_s)_{W}, \R)\rightarrow H_{ar}^1(\X_s, \R)$ 
is an isomorphism of topological groups and then it follows that 
$\underline{\Hom}(\Pi^{ab}_1(\X_s)_{W}, \R/\Z)\rightarrow H_{ar}^1(\X_s, \R/\Z)$ is continuous as well.
We obtain a canonical isomorphism 
$H_{ar}^1(\X_s, \R/\Z)\simeq \underline{\Hom}(\Pi^{ab}_1(\X_s)_{W}, \R/\Z)$
of compact abelian groups.
Moreover, the map $t$ in (\ref{coh}) is canonically identified with $s^D$.  
We obtain a canonical isomorphism between their cokernels:
$$\underline{\Hom}(\pi^{ab}_1(\X_K)_{W}, \R/\Z) 
\stackrel{\sim}{\rightarrow} H_{ar}^1(\X_K,\R/\Z).$$

\end{proof}

\subsection{Class field theory}
\begin{thm}\label{ICFT} Assume that $\X$ has good or strictly semi-stable reduction 
and that $R\Gamma_W(\X_s,\Z^c(0))$ is a perfect complex of abelian groups. 
Then there exists a unique functorial isomorphism of locally compact groups
\begin{equation*}
\underline{\rec}:H^{2d-1}_{ar}(\X_{K},\Z(d))
\stackrel{\sim}{ \longrightarrow} \pi_1^{ab}(\X_K)_{W}
\end{equation*}
inducing
$$H^{2d-1}_{et}(\X_{K},\widehat{\Z}(d))\stackrel{\sim}{ \longrightarrow} 
\pi_1^{ab}(\X_K)$$
after profinite completion.
\end{thm}

\begin{proof}
By \cite[Cor.\  5.13]{Geisser-Morin-21}, Proposition \ref{pi1corepresent},
and Pontryagin duality, we have isomorphisms  of locally compact groups
\begin{eqnarray*}
H^{2d-1}_{ar}(\X_{K},\Z(d))&\stackrel{\sim}{ \longrightarrow}& 
\underline{\Hom}(H^{1}_{ar}(\X_{K},\R/\Z),\R/\Z)\\
&\simeq& \underline{\Hom}(\underline{\Hom}(\pi_1^{ab}(\X_K)_W,\R/\Z),\R/\Z)\\
&\simeq & \pi_1^{ab}(\X_K)_{W}.
\end{eqnarray*}
To prove uniqueness, we set $(-)^{\delta}:=\mathrm{disc}(-)$ for brevity. 
Recall from Proposition \ref{ar-et-hat-et-complete} that we have canonical 
morphisms of discrete abelian groups
$$H^{i}_{ar}(\X_{K},\Z(d))^{\delta}\rightarrow H^{i}_{et}(\X_{K},\widehat{\Z}(d)).$$
The commutative square of \cite[Prop.\  5.8]{Geisser-Morin-21}
 (respectively the proof of Proposition \ref{pi1corepresent}) 
shows that the left square (respectively the right square) of the following
the diagram of discrete abelian groups
\[ \xymatrix{
H^{2d-1}_{ar}(\X_{K},\Z(d))^{\delta} \ar[r]^{\hspace{-0.8cm}\sim}\ar[d]^{} 
&\underline{\Hom}(H^{1}_{ar}(\X_{K},\R/\Z),\R/\Z)^{\delta} \ar[d]^{}
&\pi_1^{ab}(\X_K)_{W}^{\delta}\ar[d]\ar[l]_{\hspace{1.5cm}\sim}\\
H^{2d-1}_{et}(\X_{K},\widehat{\Z}(d)) \ar[r]^{\hspace{-0.8cm}\sim}
& \underline{\Hom}(H^{1}_{et}(\X_{K},\Q/\Z),\R/\Z)^{\delta} 
& \pi^{ab}_1(\X_K)^{\delta}\ar[l]_{\hspace{1.5cm}\sim}
}
\]
commutes. Here the right vertical maps is injective, the horizontal maps are all isomorphisms 	and  the composition of the upper horizontal maps is $\underline{\rec}^{\delta}$. The unicity of the map $\underline{\rec}^{\delta}$ follows, as well as the unicity of the map $\underline{\rec}$ since the functor $(-)^{\delta}:\mathrm{LCA}\rightarrow \mathrm{Ab}$ is faithful. 

Finally, we check that $\underline{\rec}$ is functorial. 
A proper map $\X\rightarrow \mathcal{Y}$ induces a commutative square
\[ \xymatrix{
H^{2d_{\X}-1}_{ar}(\X_{K},\Z(d_{\X}))^{\delta} \ar[r]^{}\ar[d]^{}&H^{2d_{\mathcal{Y}}-1}_{ar}(\mathcal{Y}_{K},\Z(d_{\mathcal{Y}}))^{\delta}\ar[d]\\
H^{2d_{\X}-1}_{et}(\X_{K},\widehat{\Z}(d_{\X})) \ar[r]^{} & H^{2d_{\mathcal{Y}}-1}_{et}(\mathcal{Y}_{K},\widehat{\Z}(d_{\mathcal{Y}}))
}
\]
as mentioned in Remark \ref{remfunctoriality}. Considering the reciprocity maps to the corresponding (discrete) fundamental groups, we obtain a 8-terms diagram. A look at this 8-terms diagram shows the result, since $H^{2d-1}_{et}(\X_{K},\widehat{\Z}(d)) \rightarrow \pi_1^{ab}(\X_K)$ is functorial for proper maps, $\pi_1^{ab}(\mathcal{Y}_K)_W\rightarrow \pi_1^{ab}(\mathcal{Y}_K)$ is injective and since $(-)^{\delta}$ is faithful. 
\end{proof}

\begin{example}
If $\X_K=\mathrm{Spec}(K)$, we recover Weil's isomorphism of locally compact abelian groups
$$K^{\times}=H^{1}_{ar}(\X_{K},\Z(1))\stackrel{\sim}{ \longrightarrow} \pi_1^{ab}(\X_K)_{W}=W_K^{ab}.$$
\end{example}

\begin{notation}\label{notgeo}
We define the locally compact groups
\begin{eqnarray*}
H^{2d-1}_{ar}(\X_K,\Z(d))^0&:=&\mathrm{Ker}\left(H^{2d-1}_{ar}(\X_K,\Z(d))\rightarrow K^{\times}\right);\\
H^{2d-1}_{ar}(\X,\Z(d))^0&:=&\mathrm{Ker}\left(H^{2d-1}_{ar}(\X,\Z(d))\rightarrow \O_{K}^{\times}\right);\\
\pi_1^{ab}(\X_K)_W^{geo}&:=&\mathrm{Ker}\left(\pi_1^{ab}(\X_K)_W\rightarrow W_K^{ab}\right).
\end{eqnarray*}
We define similarly the discrete abelian group
\begin{eqnarray*}
H^{2d-1}_{et}(\X,\widehat{\Z}(d))^0&:=&
\mathrm{Ker}\left(H^{2d-1}_{et}(\X,\widehat{\Z}(d))\rightarrow \O_{K}^{\times}\right).
\end{eqnarray*}
\end{notation}

\begin{lem}\label{finitecoker}
The cokernel of the push-forward maps from
$$H^{2d-1}_\M(\X,\Z(d)), H^{2d-1}_{et}(\X,\Z(d)), 
H^{2d-1}_{et}(\X,\Z(d))^{\widehat{}},$$
$$\widehat{H}^{2d-1}_{et}(\X,\Z(d)), H^{2d-1}_{ar}(\X,\Z(d))$$ 
to $\O^{\times}_K$ are finite, where $(-)^{\widehat{}}$ denotes the 
profinite completion.
Similarly, the cokernels of the push-forward from
$$H^{2d-1}_{\M}(\X_K,\Z(d)),
H^{2d-1}_{et}(\X_K,\Z(d)), \widehat{H}^{2d-1}_{et}(\X_K,\Z(d)),
H^{2d-1}_{ar}(\X_K,\Z(d))$$ 
to $K^{\times}$ are finite.
\end{lem}

\begin{proof}
We first note that for $\X=\Spec(\O_K)$ the first groups are isomorphic to 
$\O_K^\times$, and for $X=\Spec(K)$ the second groups are isomorphic to $K^\times$.
In view of Proposition \ref{ar-et-hat-et-complete} 
the push-forward maps $H^{2d-1}_\M(\X,\Z(d)) \to \O^{\times}_K$ and 
$H^{2d-1}_{\M}(\X_K,\Z(d))\to K^\times$ factor through the other groups, so that
it suffices to prove the Lemma for these maps.

Let $L/K$ be a finite extension such that $\X_K$ has a $L$-rational point.
By properness, we obtain a finite flat morphism $\mathrm{Spec}(\O_L)\rightarrow \X$ 
over $\mathrm{Spec}(\O_K)$.  Then the composite maps
$$\O_K^\times\subseteq 
\O_L^{\times}\simeq H^{1}_{\M}(\O_L,\Z(1))\rightarrow 
H^{2d-1}_{\M}(\X,\Z(d)) \rightarrow 
H^{1}_{\M}(\O_K,\Z(1))\simeq \O_K^{\times}$$
$$K^\times\subseteq L^{\times}\simeq H^{1}_{\M}(L,\Z(1))\rightarrow 
H^{2d-1}_{\M}(\X_K,\Z(d)) \rightarrow 
H^{1}_{\M}(K,\Z(1))\simeq K^{\times}$$
are multiplication by $[L:K$], hence have finite cokernels by the known structure
of $\O_K^{\times}$ and $K^{\times}$.
\end{proof}

\begin{prop}\label{propkey}
Assume that $\X$ has good or strictly semi-stable reduction. Then the group 
$H^{2d-1}_{ar}(\X,\Z(d))^0\simeq H^{2d-1}_{et}(\X,\widehat{\Z}(d))^0$ is finite. 
\end{prop}

\begin{proof}
By \cite[Proof of Theorem 4.16]{Geisser-Morin-21}, we have isomorphisms 
$H^{2d-1}_{et}(\X,\Q_p(d))\simeq H^1_f(G_K,H^{2d-2}(\X_{\bar{K}}, \Q_p(d)))$ and 
$H^{1}_{et}(\O_K,\Q_p(1))\simeq H^1_f(G_K,\Q_p(1))$. 
The push-forward map
$H^{2d-1}_{et}(\X,\Q_p(d))\rightarrow H^{1}_{et}(\O_K,\Q_p(1))$ can be identified 
with the map 
$H^1_f(G_K,H^{2d-2}(\X_{\bar{K}}, \Q_p(d)))\rightarrow H^1_f(G_K,\Q_p(1))$ induced 
by the trace map $H^{2d-2}(\X_{\bar{K}}, \Q_p(d))\stackrel{\sim}{\rightarrow}\Q_p(1)$, 
which is an isomorphism since $\X_K$ is geometrically connected. 
Hence the kernel $$H^{2d-1}_{et}(\X,\Z_p(d))^0:=\mathrm{Ker}\left(H^{2d-1}_{et}(\X,\Z_p(d))\rightarrow H^{1}_{et}(\O_{K},\Z_p(1)) \right)$$ 
is torsion. Since $H^{2d-1}_{et}(\X,\Z_p(d))$ is a finitely generated $\Z_p$-module, $H^{2d-1}_{et}(\X,\Z_p(d))^0$ is finite. Moreover,  the group $H^{2d-1}_{et}(\X,\Z_l(d))\simeq H^{2d-1}_{et}(\X_s,\Z_l(d))$ is finite for all $l\neq p$ and vanishes for almost all $l$, see \cite[Lemma 4.17]{Geisser-Morin-21}. Hence 
$H^{2d-1}_{et}(\X,\widehat{\Z}(d))^0$ is finite as well. Since we have 
$\mathrm{disc}(H^{2d-1}_{ar}(\X,\Z(d))^0)=H^{2d-1}_{et}(\X,\widehat{\Z}(d))^0$ by Proposition \ref{ar-et-hat-et-complete}, the locally compact group $H^{2d-1}_{ar}(\X,\Z(d))^0$ is finite hence discrete, and may therefore be identified with $H^{2d-1}_{et}(\X,\widehat{\Z}(d))^0$.
\end{proof}

\begin{prop}\label{exact-sequences}Assume that $R\Gamma_W(\X_s,\Z^c(0))$ is cohomologically bounded. Then we have a diagram with exact rows and columns
\[ \xymatrix{
H_{2}^W(\X_s,\Z)\ar[r]\ar[d]^= & H^{2d-1}_{ar}(\X,\Z(d))^0  \ar[d]\ar[r] 
&   H^{2d-1}_{ar}(\X_K,\Z(d))^0      \ar[r]^{\hspace{0.5cm}f} \ar[d]
& H_{1}^W(\X_s,\Z)^0   \ar[d]  & \\
H_{2}^W(\X_s,\Z)\ar[d]\ar[r] & H^{2d-1}_{ar}(\X,\Z(d))  \ar[d]^s\ar[r] 
&   H^{2d-1}_{ar}(\X_K,\Z(d))             \ar[r] \ar[d]
& H_{1}^W(\X_s,\Z)   \ar[d] \ar[r] & 0\\
0\ar[r]& \O^{\times}_K  \ar[r]^{} &   K^{\times} \ar[r]& \Z\ar[r] & 0 
}
\]
Moreover, we have an injective map $\cok(f)\hookrightarrow \cok(s)$ of 
finite groups. 
\end{prop}

\begin{proof} We have by duality
$$H^{2d}_{ar}(\X,\Z(d))^{\delta}\simeq H^{2d}_{et}(\X,\widehat{\Z}(d))\simeq H^{1}_{et,\X_s}(\X,\Q/\Z)^D$$
$$\simeq \left(\mathrm{Ker}(\pi^{ab}_1(\X)^D\rightarrow \pi^{ab}_1(\X_K)^D)\right)^D=0$$
since $\pi^{ab}_1(\X_K)\rightarrow \pi^{ab}_1(\X)$ is surjective because $\X$ is normal 
\cite[V, Prop.\ 8.2]{SGA1}. 
Hence the middle row is the localization sequence for arithmetic cohomology. 
The fact that the lower squares commute follows from the functoriality of 
localization sequences. The top row is defined by taking kernels; 
it is easily seen to be exact. 
The injective map $\cok(f)\hookrightarrow \cok(s)$ is given by the snake lemma. 
The fact that $\cok(s)$ is finite is Lemma in \ref{finitecoker}.
\end{proof}

\begin{cor}\label{ICFT0}
Assume that $\X$ has good or strictly semi-stable reduction and that 
$R\Gamma_W(\X_s,\Z^c(0))$ is a perfect complex of abelian groups.
Then the map $\underline{\rec}$ induces an isomorphism
\begin{equation}
\underline{\rec}^0:H^{2d-1}_{ar}(\X_{K},\Z(d))^0
\stackrel{\sim}{ \longrightarrow} \pi_1^{ab}(\X_K)^{geo}_W
\end{equation}
of finitely generated abelian groups of ranks 
$$\mathrm{rank}_{\Z}(H_{ar}^{2d-1}(\X_K,\Z(d))^0)=
\mathrm{rank}_{\Z}(H_1^W(\X_s,\Z))-1.$$
\end{cor}

\begin{proof}
The fact that $\underline{\rec}^0$ is an isomorphism follows from the fact that
isomorphism $\underline{\rec}$ is functorial. Finite generation and the 
statement about the rank of $H^{2d-1}_{ar}(\X_K,\Z(d))^0$ follows 
from Proposition \ref{propkey} and Proposition \ref{exact-sequences}
because the map $H_{1}^W(\X_s,\Z)\rightarrow H_{1}^W(s,\Z)$ has finite cokernel, 
as can be seen by considering a closed point of $\X_s$.
\end{proof}

\section{Comparison to the classical reciprocity map}
In this section we study the composition 
$H_{et}^{2d-1}(\X_K,\Z(d))\rightarrow
\widehat{H}_{et}^{2d-1}(\X_K,\Z(d))\rightarrow{H}_{ar}^{2d-1}(\X_K,\Z(d))$,
generalizing the reciprocity map of Saito \cite{Saito85} for curves. 

\subsection{The map $H_{et}^{2d-1}(\X_K,\Z(d))\rightarrow
\widehat{H}_{et}^{2d-1}(\X_K,\Z(d)) $.}

The following lemma follows by an easy diagram chase.

\begin{lem}\label{lem-cokertokersequence}
Consider a diagram with exact rows of abelian groups 
\[ \xymatrix{
A^{i}\ar[r] \ar[d]^{\simeq}&B^{i} \ar[r] \ar[d]^{f^{i}} &C^{i} \ar[r]^{}\ar[d]^{g^{i}}&A^{i+1}\ar[r] \ar[d]^{\simeq}&B^{i+1} \ar[r] \ar[d]^{f^{i+1}} &C^{i+1} \ar[r]^{}\ar[d]^{g^{i+1}}&A^{i+2}\ar[d]^{\simeq}\\
A'^{i}\ar[r] &B'^{i} \ar[r]  &C'^{i} \ar[r]^{}&A'^{i+1}\ar[r] &B'^{i+1} \ar[r] &C'^{i+1}\ar[r] &A'^{i+2},\\
}
\]
where the maps $A^j\rightarrow A'^j$ are isomorphisms. Then there is an induced exact sequence
$$0\rightarrow \cok(f^{i})\rightarrow \cok(g^{i})\rightarrow 
\ker(f^{i+1})\rightarrow\ker (g^{i+1})\rightarrow 0.$$
\end{lem}

\begin{lem}\label{lemreduct}
Assume that the group $H^{2d-1}_{et}(\X,\widehat{\Z}(d))^0$ is finite. 
Then one has $TH^{2d}_{et}(\X,\Z(d))=0$ and the map 
\begin{equation}\label{forcoker}
H_{et}^{2d-1}(\X,\Z(d))\rightarrow H_{et}^{2d-1}(\X,\Z(d))^{\widehat{}}
\end{equation}
is surjective, where $(-)^{\widehat{}}$ denotes the (naive) profinite completion.
In particular, 
$H_{et}^{2d-1}(\X,\Z(d))\rightarrow H_{et}^{2d-1}(\X,\widehat\Z(d))$
is surjective.
\end{lem}


\begin{proof}
Consider the morphism of short exact sequences:
\[ \xymatrix{
0\ar[r] & H^{2d-1}_{et}(\X,\Z(d))^{\widehat{}}  \ar[d]^{a} \ar[r] &   H^{2d-1}_{et}(\X,\widehat{\Z}(d))             \ar[r] \ar[d]^{b}& T H^{2d}_{et}(\X,\Z(d))   \ar[d] \ar[r] & 0\\
0\ar[r]& H^{1}_{et}(\O_K,\Z(1))^{\widehat{}}  \ar[r]^{\sim} &   H^{1}_{et}(\O_K,\widehat{\Z}(1))          \ar[r]    & 0     \ar[r] & 0 
}
\]
The group $\cok(a)$ is finite by Lemma \ref{finitecoker}. 
Since  $\ker(b)$ is finite by assumption, the Tate module 
$T H^{2d}_{et}(\X,\Z(d))$ is finite by the snake lemma. 
Since $T H^{2d}_{et}(\X,\Z(d))$ is also torsion-free, it is in fact trivial.

We obtain an isomorphism of finite groups $\ker(a)\simeq \ker(b)$. 
Now we consider the morphism of exact sequences
\[ \xymatrix{
0\ar[r]  & H^{2d-1}_{et}(\X,\Z(d))^{0}  \ar[d] \ar[r] &   H^{2d-1}_{et}(\X,\Z(d))  \ar[r] \ar[d]^{(\ref{forcoker})}& \O^{\times}_K  \ar[d]^{=} \ar[r] & C\ar[d]\ar[r] &0\\
0\ar[r]  & 
\ker(a) \ar[r] &   H^{2d-1}_{et}(\X,\Z(d))^{\widehat{}}  \ar[r] & \O^{\times}_K   \ar[r]& C' \ar[r] &0
}
\]
where $C$ and $C'$ are finite by Lemma \ref{finitecoker}. Since $\ker(a)$
is finite, it follows from Lemma \ref{lem-cokertokersequence} that 
$\cok(\ref{forcoker})$ is finite as well. But $\cok(\ref{forcoker})$ 
is also divisible, since $\cok(A\rightarrow \widehat{A})$ is divisible for 
any abelian group $A$. Hence $\cok(\ref{forcoker})$ vanishes. 
\end{proof}

\begin{lem} \label{lemvanishingChow}
One has $CH^d(\X_{\O_{K^{un}}})=0$.
\end{lem}

\begin{proof}
We set $\X^{un}:=\X_{\O_{K^{un}}}$. We denote the closed fiber of $\X^{un}$ by $\X_{\bar{s}}$. 
It suffices to show that the image of every closed point $x$ of 
$\X^{un}$ in $CH_0(\X^{un})$ vanishes. Note that $\O_{K^{un}}$ is an henselian discrete valuation ring.
Since $\X^{un}$ is proper over $\O_{K^{un}}$, $x$ lies in the
closed fiber $\X_{\bar{s}}$. By \cite[Lemma 7.2]{Saito-Sato-10},
there exists a regular integral closed subscheme $Z$ of $\X$ of dimension $1$
containing $x$ which meets the generic fiber. Now $Z$ is proper over 
the base $\O_{K^{un}}$, hence finite. Since the base is henselian, this implies
that $Z$ is the spectrum of an henselian discrete valuation ring.
But for the spectrum $Z$ of a discrete valuation ring with function field $K$
and closed point $x$ 
the map $CH_0(K,1) \simeq K^{\times}\rightarrow \Z=CH_0(x)$ in 
the localization sequence is surjective,
hence its cokernel $CH_0(Z)$ vanishes. Thus the map $\Z=CH_0(x)\rightarrow CH_0(\X^{un})$ is the zero map, since it factors through $CH_0(Z)$.
\end{proof}

\begin{thm}\label{neutheorem}
Assume that the group $H^{2d-1}_{et}(\X,\widehat{\Z}(d))^0$ is finite. 
Then there are exact sequences of abelian groups
$$0\rightarrow  D \rightarrow  H_{et}^{2d-1}(\X_K,\Z(d))\rightarrow 
\widehat{H}_{et}^{2d-1}(\X_K,\Z(d))^{\delta}\rightarrow 0$$
$$0\rightarrow  D \rightarrow  H_{et}^{2d-1}(\X_K,\Z(d))^0\rightarrow 
\widehat{H}_{et}^{2d-1}(\X_K,\Z(d))^{0\delta}\rightarrow 0,$$
where $D$ is divisible, and 
$$\Tor D\simeq \cok (H_{et}^{2d-2}(\X_K,\Z(d))\to
\widehat{H}_{et}^{2d-2}(\X_K,\Z(d))^\delta)\otimes\Q/\Z,$$ 
\end{thm}

Note that by Proposition \ref{propkey} the hypothesis of the Theorem 
is satisfied if $\X_K$ has good or strictly semi-stable reduction. 
If $A$ is an abelian group, we denote by $UA:=\cap_{m} m\cdot A$ the subgroup of 
all divisible elements in $A$, and by $A_{div}$ the maximal divisible 
subgroup of $A$.

\begin{proof}
Let $\ker^{i}(\X)$ and $\cok^{i}(\X)$ be the kernel and cokernel of the 
map $H_{et}^{i}(\X,\Z(d))\rightarrow \widehat{H}_{et}^{i}(\X,\Z(d))^{\delta}$,
respectively. 
We define similarly $\ker^{i}(\X_{K})$ and $\cok^{i}(\X_{K})$. 
The morphism of long exact sequences
{\small{
\[ \xymatrix{
\cdots H^{2d-1}_{et}(\X_s,Ri^{!}\Z(d))\ar[r] \ar[d]^{\simeq}&
H^{2d-1}_{et}(\X,\Z(d)) \ar[r] \ar[d] &
H^{2d-1}_{et}(\X_{K},\Z(d)) \ar[r]^{}\ar[d]^{}&
H^{2d}_{et}(\X_s,Ri^{!}\Z(d))\ar[d]^{\simeq}\cdots\\
\cdots H^{2d-1}_{et}(\X_s,Ri^{!}\Z(d))\ar[r] &
\widehat{H}^{2d-1}_{et}(\X,\Z(d))^{\delta} \ar[r] &
\widehat{H}^{2d-1}_{et}(\X_{K},\Z(d))^{\delta} \ar[r]^{}&
H^{2d}_{et}(\X_s,Ri^{!}\Z(d))]\cdots
}
\]
}}
yields by Lemma \ref{lem-cokertokersequence} the exact sequence
\begin{equation}\label{isoCok}
0\rightarrow \cok^{2d-1}(\X) \rightarrow \cok^{2d-1}(\X_{K}) \rightarrow 
\ker^{2d}(\X) \rightarrow \ker^{2d}(\X_{K})  \rightarrow 0.
\end{equation}
First we note that $\cok^{2d-1}(\X)=0$ because 
$$H_{et}^{2d-1}(\X,\Z(d))\rightarrow 
H_{et}^{2d-1}(\X,\widehat\Z(d))\simeq \widehat 
H_{et}^{2d-1}(\X,\Z(d))^\delta$$
is surjective by Lemma \ref{lemreduct} and Proposition 
\ref{ar-et-hat-et-complete}.

Since the map 
$H_{et}^{2d}(\X,\Z(d))^{\widehat{}}\rightarrow H_{et}^{2d}(\X,\widehat{\Z}(d))$
is injective, we have $\ker^{2d}(\X)= UH^{2d}_{et}(\X,\Z(d))$. 
By \cite[Cor.\  4]{Milne15}, the group $UH^{2d}_{et}(\X,\Z(d))$ is uniquely
divisible because $TH^{2d}_{et}(\X,\Z(d))=0$ and 
$H^{2d}_{et}(\X,\Z(d))/m\subset H^{2d}_{et}(\X,\Z/m(d))$ is finite for any $m$. 
Moreover, 
$$H^{2d}_{et}(\X_{\O_{K^{un}}},\Z(d))\simeq H^{2d}_{Zar}(\X_{\O_{K^{un}}},\Z(d))
\simeq CH^d(\X_{\O_{K^{un}}})=0$$
by Lemma \ref{lemvanishingChow}, hence
the Hochschild-Serre spectral sequence gives a short exact sequence
\begin{multline*}
\stackrel{d_2^{0,2d-1}}{\longrightarrow} 
H^2(G_s,H^{2d-2}_{et}(\X_{\O_{K^{un}}},\Z(d)))
\rightarrow H^{2d}_{et}(\X,\Z(d))\\
\rightarrow H^1(G_s,H^{2d-1}_{et}(\X_{\O_{K^{un}}},\Z(d)))\rightarrow 0
\end{multline*}
since $G_s$ has strict cohomological dimension $2$. Since Galois cohomology is 
torsion in degrees $>0$, the abelian group $H^{2d}_{et}(\X,\Z(d))$ is torsion as
well. Therefore, $\ker^{2d}(\X)=UH^{2d}_{et}(\X,\Z(d))$ 
is both uniquely divisible and torsion, hence it vanishes.
In view of (\ref{isoCok}), we obtain $\cok^{2d-1}(\X_{K})=0$.

To obtain the second exact sequence it suffices to consider the kernels of
the commutative diagram with exact rows
\[ \xymatrix{
0\ar[r]& D\ar[d] \ar[r]&H^{2d-1}_{et}(\X_K,\Z(d))\ar[r]\ar[d] &
 \widehat{H}^{2d-1}_{et}(\X_K,\Z(d))^{\delta}  \ar[d]\ar[r] &   0\\
0\ar[r]&0           \ar[r]       &K^{\times}  \ar[r]                                                & K^{\times}                                                                           \ar[r] &    0                    
}
\]
Applying Lemma \ref{lem-cokertokersequence} 
to the coefficient sequences induced by 
$0\to \Z\stackrel{\times m}{\longrightarrow}\Z\to \Z/m\to 0$, and using that  
$$R\Gamma_{et}(\X_{K},\Z(d))\otimes^L\Z/m\Z \rightarrow 
R\widehat{\Gamma}_{et}(\X_{K},\Z(d))^{\delta} \otimes^L\Z/m\Z$$
is an equivalence by definition,  we obtain an exact sequence
$$ 0\to \cok^{i-1}(\X_{K})\stackrel{\times m}{\longrightarrow} \cok^{i-1}(\X_{K})\to 
\ker^{i}(\X_{K})\stackrel{\times m}{\longrightarrow} \ker^{i}(\X_{K})\to 0,$$
which implies that $\ker^{i}(\X_{K})$ is divisible and 
and ${}_m \ker^{i}(\X_{K})\simeq \cok^{i-1}(\X_{K})/m$ for all $i$ and $m$.
In particular, $D$ is divisible and in the colimit we obtain
$\Tor D\simeq  \cok^{2d-2}(\X_{K})\otimes\Q/\Z$.
\end{proof}

\subsection{The maps $\widehat{H}_{et}^{2d-1}(\X_K,\Z(d))
\rightarrow{H}_{ar}^{2d-1}(\X_K,\Z(d)) $.}

\begin{prop}\label{propforgap}
Assume that $\X$ has good or strictly semi-stable reduction. Then the group 
$\widehat{H}^{2d-1}_{et}(\X_K,\Z(d))^0$ is torsion. 
\end{prop}

\begin{proof}
Replacing $R\Gamma_{ar}(-,\Z(d))$ with $R\widehat{\Gamma}_{et}(-,\Z(d))$ 
in the diagram of Proposition \ref{exact-sequences}, we obtain an exact sequence
$$\widehat{H}^{2d-1}_{et}(\X,\Z(d))^0\rightarrow 
\widehat{H}^{2d-1}_{et}(\X_K,\Z(d))^0\rightarrow H^0_{et}(\X_s,\Z^c(0))^0$$
where $H^0_{et}(\X_s,\Z^c(0))^0$ is the kernel of the push-forward map 
$H^0_{et}(\X_s,\Z^c(0))\rightarrow H^0_{et}(s,\Z^c(0))\simeq \Z$. 
Since $\widehat{H}^{2d-1}_{et}(\X,\Z(d))^0\simeq H^{2d-1}_{ar}(\X,\Z(d))^0$ 
is finite by Proposition \ref{propkey}, it is enough to show that 
$H^0_{et}(\X_s,\Z^c(0))^0$ is torsion, i.e. that the degree map 
$CH_0(\X_s)_{\Q}\simeq H^0(\X_s,\Q^c(0))\rightarrow \Q\simeq H^0(s,\Q^c(0))$ 
is an isomorphism. But every element of $CH_0(\X_s)_{\Q}$ is supported on a 
curve, and in this case the result is well-known, see for example 
\cite[Thm.\ 3.1, Prop.\ 6.2]{Geisser-Crelle}.
\end{proof}

\begin{prop}\label{proplast}
Assume that $\X$ has good or strictly semi-stable reduction and 
and $R\Gamma_W(\X_s,\Z^c(0))$ is a perfect complex of abelian groups. 
Then there is an exact sequence
$$0\to \widehat D \to \widehat{H}^{2d-1}_{et}(\X_K,\Z(d))^0
\stackrel{\tau}{\longrightarrow} \Tor H^{2d-1}_{ar}(\X_K,\Z(d))^0\to 0$$
with $\widehat D$ is the maximal divisible subgroup of 
$\widehat{H}^{2d-1}_{et}(\X_K,\Z(d))^0$, a torsion quotient of 
$CH_0(\X_s,2)_{\Q}$.
\end{prop}

\begin{proof}
Since  $H^{2d-1}_{ar}(\X_K,\Z(d))^0$ is finitely generated 
by Corollary \ref{ICFT0}, the maximal divisible subgroup of 
$\widehat{H}^{2d-1}_{et}(\X_K,\Z(d))^0$ is contained in $\ker \tau$.
Moreover, $\widehat{H}^{2d-1}_{et}(\X_K,\Z(d))^0$
is torsion by Proposition \ref{propforgap}, so that $\hat D$ is
torsion and the image of $\tau$ 
is contained in the torsion subgroup $\Tor H^{2d-1}_{ar}(\X_K,\Z(d))^0$. 
In order to show that $\tau$ is surjective it suffices to show 
that $\cok \tau$ is torsion free, and to determine $\ker \tau$ it suffices to 
show that it is a quotient $CH_0(\X_s,2)_{\Q}$. 
Consider the morphism of short exact sequences of discrete abelian groups:
\[ \begin{CD}
0@>>> \widehat{H}^{2d-1}_{et}(\X_K,\Z(d))^{0\delta}@>>>
\widehat{H}^{2d-1}_{et}(\X_K,\Z(d))^{\delta}@>>> K^{\times}\\
@.@VVV @VVV@|\\
0@>>> H^{2d-1}_{ar}(\X_K,\Z(d))^{0\delta}@>>> H^{2d-1}_{ar}(\X_K,\Z(d))^{\delta}
@>>> K^{\times}.
\end{CD}
\]
We see that the  kernels of the left two maps are isomorphic, and that 
the cokernel of the left vertical map is subgroup of the cokernel of the 
middle vertical map.
Hence the long exact sequence associated with the cofiber sequence 
\eqref{descent-triangle-top-redefined-n=d} 
$$CH_0(\X_s,2)_{\Q}\rightarrow \widehat{H}^{2d-1}_{et}(\X_K,\Z(d))^{\delta} 
\rightarrow H^{2d-1}_{ar}(\X_K,\Z(d))^{\delta}\rightarrow CH_0(\X_s,1)_{\Q}$$
shows that the map 
$\widehat{H}^{2d-1}_{et}(\X_K,\Z(d))^0\rightarrow H^{2d-1}_{ar}(\X_K,\Z(d))^0$ 
has kernel a quotient of $CH_0(\X_s,2)_{\Q}$ and torsion free cokernel. 
\end{proof}

\begin{cor}
Under the hypothesis of the theorem 
there is an exact sequence of abelian groups
$$0\rightarrow  \mathcal K \rightarrow  H_{et}^{2d-1}(\X_K,\Z(d))^0\rightarrow 
\Tor H^{2d-1}_{ar}(\X_K,\Z(d))^0\to 0$$
where $\mathcal K\simeq D\oplus \hat D$ is the maximal divisible subgroup
of $H_{et}^{2d-1}(\X_K,\Z(d))^0$ and
$\Tor \mathcal K$ is the direct sum of 
$\cok (H_{et}^{2d-2}(\X_K,\Z(d))\to\widehat{H}_{et}^{2d-2}(\X_K,\Z(d)))\otimes\Q/\Z$
and a torsion quotient of $CH_0(\X_s,2)_{\Q}$.
\end{cor}

\begin{proof}
This follows from Theorem \ref{neutheorem} and Proposition \ref{proplast}
because for a composition 
$g\circ f$ of surjections we have an exact sequence 
$0\to \ker f\to \ker (g\circ f)\to \ker g \to 0$, which splits because
$D=\ker f$ is divisible. The divisible group $\mathcal K$ is maximal divisible
because $H^{2d-1}_{ar}(\X_K,\Z(d))^0$ is finitely generated.
\end{proof}

\section{Comparison to the motivic reciprocity map}
In this section we compare the composition 
$H_\M^{2d-1}(\X_K,\Z(d))\to H_{ar}^{2d-1}(\X_K,\Z(d)) $
to the classical reciprocity map. 
\subsection{Kato-homology}
Only for this section we are using the work 
of Jannsen-Saito and Kerz-Saito on Kato-homology in order to be able to make
statements about $H_\M^{2d-1}(\X_K,\Z(d))$, see \cite{Kato}, \cite{Kerz-Saito},
and \cite{Geisser10}. 

For $\X$ a separated scheme
of finite type over $\O_K$, we define the Kato complex $KC(\X)$ to be 
$\cone\large(\Z^c(\X)\to R\Gamma_\et(\X,\Z^c)\large)[-1]$, 
where $\Z^c=\Z^c(0)$ is Bloch's cycle complex of cycles of relative dimension $0$
viewed as a complex of \'etale sheaves. For an abelian group $A$
we let $KH_{i}(\X,A)$ be the homology of $KC(\X)\otimes^L A$ so that there is 
an exact sequence
$$ \cdots\to  KH_{i+2}(\X,A)\to  CH_0(\X,i,A)\to H_i^\et(\X,A)\to 
KH_{i+1}(\X,A)\to \cdots.$$
By \cite[Cor.\ 5.2, Cor.\ 5.5, Cor.\ 7.7]{Geisser10}, $KC(\X)\otimes^L \Z/m$ is 
quasi-isomorphic to the Kato-complexes discussed in \cite{Kato} and 
\cite{Kerz-Saito}. Kerz-Saito proved the following theorem 
\cite[Thm.\ 8.1]{Kerz-Saito}, which was conjectured by Kato 
\cite[Conj.\ 0.3, 5.1]{Kato},

\begin{thm}
1) If $Y$ is regular and connected scheme, proper over a finite field, 
then 
$KH_i(Y,\Z/m)=0$ for $i>0$ and $KH_0(Y,\Z/m)\cong \Z/m$ if either 
$m$ is prime to $p$, or if $i\leq 4$ and $Y$ is projective.

2) If $\X$ is regular and connected scheme, proper over $\O_K$,  
then $KH_i(\X,\Z/m)=0$ for all $i$ if either $m$ prime to $p$,
or if $i\leq 4$ and $\X$ is projective.
\end{thm}

Since $\Z^c(\X)\otimes\Q \to R\Gamma_\et(\X,\Z^c)\otimes\Q$ is
a quasi-isomorphism, this implies:

\begin{cor}\label{usekato}
If $\X$ is a regular and connected scheme, projective over $\O_K$,  
then $KH_i(\X,\Z)=0$ for $i\leq 3$.
In particular, $H^{i}_\M(\X,\Z(d))\cong H^{i}_\et(\X,\Z(d))$
for $i\geq 2d-1$, and $H^{2d-2}_\M(\X,\Z(d))\to H^{2d-2}_\et(\X,\Z(d))$
is surjective.
\end{cor}

Now let $\X$ be a regular scheme, proper scheme over $\O_K$ with generic fiber
$\X_K$. Note that $KC(\X_K)$ is
$\cone(\Z^c(-1)(\X_K)\to R\Gamma_\et(\X_K,\Z^c(-1))[1]$ 
because a scheme of dimension 
$n$ over $\O_K$ has generic fiber of dimension $n-1$
over $K$. We obtain an exact sequence 
$$ \cdots\to  KH_{i+1}(\X_K,A)\to  CH_{-1}(\X_K,i,A)\to H_{i-2}^\et(\X_K,A(-1))\to 
KH_{i}(\X_K,A)\to \cdots$$
or equivalently, for $d=\dim \X$,
\begin{equation}\label{katohom}
\cdots\to  KH_{i+1}(\X_K,A)\to  H^{2d-i}_\M(\X_K,A(d))\to H^{2d-i}_\et(\X_K,A(d))
\to KH_{i}(\X_K,A)\to \cdots.
\end{equation}
From the localization sequences we obtain a long exact sequence 
$$\cdots\to  KH_{i+1}(\X,A) \to KH_i(\X_K,A)\to KH_i(\X_s,A)\to KH_i(\X,A)\to \cdots$$
and the Corollary implies that for $\X$ a regular and connected
scheme, projective over $\O_K$, we have 
$$KH_i(\X_K,A)\cong KH_i(\X_s,A)$$
for $i\leq 2$. 

\begin{example}
Let $\X$ be $\Spec\; \O_K$. Then the exact sequence for $\Spec\; k$
becomes $CH_0(k)\cong H_0^\et(k,\Z)\cong \Z$, $H_{-1}^\et(k,\Z)\cong KH_0(k,\Z)=0$,
and $H_{-2}^\et(k,\Z)\cong KH_{-1}(k,\Z)\cong \Q/\Z$. For $\X_K$ we obtain
$H^1_\M(K,\Z(1)\cong H^1_\et(K,\Z(1))\cong K^\times$, 
$H^2_\M(K,\Z(1)\cong H^2_\et(K,\Z(1))\cong 0$, and
$H^3_\et(K,\Z(1))\cong KH_{-1}(K,\Z)\cong \Q/\Z$.
\end{example}

\subsection{Homology of the closed fiber} 
For a variety $Y$ over a finite field, the groups 
$H^W_i(Y,\Z^c)=H^{1-i}R\Gamma_W(Y,\Z^c)$ 
have been studied in \cite{Geisser-Crelle}, and for proper $Y$
in \cite{Geisser-Schmidt-17}. They are finitely generated for all 
$i$ and $Y$ if and only if the groups $CH_0(X,i)$ are torsion for all $i > 0$ 
and all smooth and proper $X$ \cite[Prop.\ 4.2]{Geisser-Crelle}.
If $Y$ is connected and proper, then the degree map induces an isomorphism 
$H_0^W(Y,\Z^c)\cong \Z$
\cite[Prop.\ 4.3]{Geisser-Crelle}.
The homology groups $H_i^K(Y,A)$ of the complex $C^K(Y)\otimes A$ defined in 
\cite[Def. 5.1]{Geisser-Crelle} are isomorphic to $KH_i(Y,\Z/m)$ for $A=\Z/m$.
For curves finite generation is known:

\begin{prop}\cite[Thm.\ 6.2, Prop.\ 6.3]{Geisser-Crelle}\label{katofinite}
If $C$ is a curve,  then $H^W_0(C, \Z^c)\cong \Z^{\pi_0(C)}$, there is a
short exact sequence of finitely generated groups
$$ 0\to CH_0(C)\to H^W_1(C, \Z^c ) \to H^K_1 (C, \Z) \to 0,$$
an isomorphism $CH_0(C, 1) \cong H^W_2(C, \Z^c)$ of finitely generated groups, 
and $H^W_{i}(C, \Z^c)$ vanishes for $i > 2$. In particular, $R\Gamma_W(C,\Z^c(0))$
is perfect. 
\end{prop}

We are also using the following finite generation statement for Chow groups: 

\begin{prop}\label{Chowfinite}
For any proper scheme $Y$ over a finite field,
the group $CH_0(Y)^0=\ker( CH_0(Y)\to \Z^{\pi_0(Y)})$ is finite.
\end{prop}

\begin{proof}
For an abstract blow-up diagram
$$\begin{CD}
T'@>>> Y'\\
@VVV@VVV\\
T@>>>Y
\end{CD}$$
we obtain a map of exact sequences
$$\begin{CD}
CH_0(T')@>>> CH_0(T)\oplus CH_0(Y')@>>> CH_0(Y)@>>> 0\\
@VVV@VVV@VVV\\
\Z^{\pi_0(T')}@>>> \Z^{\pi_0(T)}\oplus \Z^{\pi_0(Y')}@>>> \Z^{\pi_0(Y)}
@>>> 0 .
\end{CD}$$
Since the vertical maps have finite cokernel, finiteness of the kernel
of the left two maps implies finiteness of the kernel of the right map.
Using Chow's Lemma and then normalizing to obtain $Y'$, we can argue by 
induction on the dimension of $Y$, to assume that $Y$ is normal and
projective, and then that $Y$ is connected. 
Since the abelianized geometric fundamental group 
$\pi_1^{ab}(Y)^{geo}$ is finite for normal $Y$, it suffices to show
that $CH_0^0(Y)\to \pi_1^{ab}(Y)^{geom}$ is injective. 
Now an argument of Colliot-Th\'el\`ene
\cite[\S 9]{Kato-Saito}, using \cite{Seidenberg} to produce hyperplane
sections, shows that it suffices to prove this for surfaces,
in which case we can use a resolution of singularities $Y'\to Y$ to
reduce to the smooth and proper case, which is known by 
Kato-Saito \cite[Thm.\ 1]{Kato-Saito}.
\end{proof}

\subsection{The reciprocity map 
$H_\M^{2d-1}(\X_K,\Z(d))\to H_{ar}^{2d-1}(\X_K,\Z(d)) $.}

\begin{thm}
If $\X_s$ is a strict normal crossing scheme and $R\Gamma_W(\X_s,\Z^c(0))$ is 
a perfect complex of abelian groups, then the reciprocity map 
$$H^{2d-1}_\M(\X_K,\Z(d))^0\to H^{2d-1}_{ar}(\X_K,\Z(d))^0$$
has finite image, and the prime to $p$ torsion part of the kernel 
is, up to finite groups, the image of $CH_0(\X_s,1)\to H^{2d-1}_\M(\X,\Z(d))$
tensored with $\Q/\Z'$.
\end{thm}

The groups $CH_0(Y,i)$ are expected to be finitely generated for any
variety $Y$ over a finite field, see Prop. \ref{katofinite} for the 
case of curves. In this case, we obtain that
the prime to $p$ torsion part of the kernel has corank at most
the rank of  $CH_0(\X_s,i)$. 

\proof
We have the following commutative diagram:
$$\begin{CD}\
CH_0(\X_s,1)@>w>> H^{2d-1}_\M(\X,\Z(d))^0@>>>H^{2d-1}_\M(\X_K,\Z(d))^0@>>> 
\overset{finite}{CH_0(\X_s)^0}\\
@Vf_1VV@VsVV@VrVV@Vf_0VV\\
H_2^W(\X_s,\Z^c)@>>> \underset{finite}{H^{2d-1}_{ar}(\X,\Z(d))^0}
@>>>H^{2d-1}_{ar}(\X_K,\Z(d))^0@>f>> 
H_1^W(\X_s,\Z^c)^0
\end{CD}$$
The lower row is exact and $f$ has finite cokernel by Proposition 
\ref{exact-sequences}. The group $CH_0(\X_s)^0$ is finite by Proposition 
\ref{Chowfinite}. The group $H^{2d-1}_{ar}(\X,\Z(d))^0$ is finite by 
Proposition \ref{propkey} and $H^{2d-1}_{ar}(\X_K,\Z(d))^0$ is finitely 
generated of the same rank as that of $H_1^W(\X_s,\Z^c)^0$
by Corollary \ref{ICFT0}. It follows that $r$ has finite image, and 
that modulo the Serre subcategory of finite group, there is an exact sequence
$$CH_0(\X_s,1)\stackrel{w}{\to} H^{2d-1}_\M(\X,\Z(d))^0\to \ker r \to 0.$$
The $\Tor_i(-,\Q/\Z)$ sequence becomes (away from $p$)
$$ \Tor H^{2d-1}_\M(\X,\Z(d))^0 \to \Tor\ker r\to \im\;w\otimes\Q/\Z\to 
H^{2d-1}_\M(\X,\Z(d))^0\otimes\Q/\Z.$$
By \cite[Lemma 4.16]{Geisser-Morin-21}, 
the group  $H^i_\et(\X,\hat\Z'(d))$ is finite for all $i$, and using the coefficient
sequence this implies that $H^{i-1}_\et(\X,\Q/\Z'(d))$ is finite, 
hence the divisible subgroup $H^{i-1}_\et(\X,\Z(d))\otimes \Q/\Z'$ vanishes
and the prime to $p$-torsion of $H^i_\et(\X,\Z(d))$ is finite.
By Cor. \ref{usekato}, $H^i_\M(\X,\Z(d))\cong H^i_\et(\X,\Z(d))$ for any $i\geq 2d-1$, 
hence the same is true for motivic cohomology. 
Thus the outer terms in the short exact sequence are finite, and the prime to 
$p$ torsion of $\ker r$ is isomorphic to the $\im\; w\otimes\Q/\Z'$.
\endproof

\section{Curves}

Suppose that $\X/\O_K$ is a relative curve with good or strictly semi-stable 
reduction. 
By Proposition \ref{katofinite}, the hypothesis of Proposition \ref{proplast} 
is satisfied and we have  $CH_0(\X_s,2)_{\Q}$, hence Proposition \ref{proplast}
gives an isomorphism 
\begin{equation}\label{curveaaa}
\widehat{H}^{3}_{et}(\X_K,\Z(2))^0\simeq  \Tor H^{3}_{ar}(\X_K,\Z(2))^0,
\end{equation}
and the groups are finite because $H^{3}_{ar}(\X_K,\Z(2))^0$ is finitely generated.
We use this to relate our results to the result of
Saito \cite{Saito85}.

\begin{thm}\label{thmRec1}
Let $\X/\O_K$ be a relative curve with has good or strictly semi-stable reduction. 
Then the reciprocity map
$$\rec:SK_1(\X_K)\rightarrow \pi_1^{ab}(\X_K)$$ 
factors as follows:
\begin{eqnarray*}
SK_1(\X_K)\simeq H^{3}_{et}(\X_{K},\Z(2))&\stackrel{s}{\twoheadrightarrow}& 
\widehat{H}^{3}_{et}(\X_{K},\Z(2))\\
&\stackrel{i}{\hookrightarrow}& H^{3}_{ar}(\X_{K},\Z(2))\stackrel{\sim}{\rightarrow} \pi_1^{ab}(\X_K)_{W}\rightarrow \pi_1^{ab}(\X_K).
\end{eqnarray*}
The map $i$ is injective,  
$s$ is surjective, and the kernel of $s$ is the 
maximal divisible subgroup of $SK_1(\X_K)$. Its torsion is isomorphic 
to $$\cok (H_{et}^{2}(\X_K,\Z(2))\to
\widehat{H}_{et}^{2}(\X_K,\Z(2))^\delta)\otimes\Q/\Z.$$\end{thm}

\begin{proof}
We have canonical isomorphisms
$$SK_1(\X_K)\simeq H_\M^3(\X_K,\Z(2))\simeq H_{et}^3(\X_K,\Z(2))$$
and the reciprocity map $\rec$ coincides with the composition
$$SK_1(\X_K)\simeq H_{et}^3(\X_K,\Z(2))\rightarrow 
H_{et}^{3}(\X_K,\widehat{\Z}(2))\stackrel{\sim}{\rightarrow} \pi_1^{ab}(\X_K) $$
where the last isomorphism is given by \'etale duality with finite coefficients. 
The middle map
factors as
\begin{equation*}
H_{et}^3(\X_K,\Z(2))\rightarrow \widehat{H}_{et}^3(\X_K,\Z(2))\rightarrow 
H_{ar}^3(\X_K,\Z(2))\rightarrow H_{et}^{3}(\X_K,\widehat{\Z}(2))
\end{equation*}
by construction, where the maps are 
continuous if $H_{et}^3(\X_K,\Z(2))$ is endowed with the discrete topology and 
both $H_{et}^{3}(\X_K,\widehat{\Z}(2))$ and $\pi_1^{ab}(\X_K) $ are endowed 
with their natural profinite topology.  
As in the proof of Theorem \ref{ICFT}, we have a commutative square
\[ \xymatrix{
 H_{ar}^{3}(\X_K,\Z(2))\ar[r]^{\sim}\ar[d]^{}& \pi_1^{ab}(\X_K)_W\ar[d]\\
 H_{et}^{3}(\X_K,\widehat{\Z}(2))\ar[r]^{\sim}& \pi_1^{ab}(\X_K)
}
\]
which we now consider as a  commutative square of locally compact groups. 
The factorization of the reciprocity map claimed in the theorem follows. 
The surjectivity of $s$ and the description of its kernel
is Theorem \ref{neutheorem}.
The injectivity of the map 
$\widehat{H}^{3}_{et}(\X_{K},\Z(2))\rightarrow H^{3}_{ar}(\X_{K},\Z(2))$ 
follows from \eqref{curveaaa}. 
\end{proof}

Removing the contribution coming from the base, we can obtain a version
with finitely generated groups. 
Define 
$$SK_1(\X_K)^0:=\mathrm{Ker}(SK_1(\X_K)\rightarrow K^{\times});$$
and similarly
$$\pi_1^{ab}(\X_K)^{geo}:=\mathrm{Ker}(\pi_1^{ab}(\X_K)\rightarrow G_K^{ab}).$$

\begin{thm}\label{thmreccurve}
Assume that $\X$ is a relative with good or strictly semi-stable reduction. 
Then the reciprocity map
$$SK_1(\X_K)^0\rightarrow \pi_1^{ab}(\X_K)^{geo}$$ 
factors as
$$SK_1(\X_K)^0\stackrel{s^0}{\twoheadrightarrow} 
\widehat{H}^{3}_{et}(\X_{K},\Z(2))^0\stackrel{i^0}{\hookrightarrow} 
H^{3}_{ar}(\X_{K},\Z(2))^0\stackrel{\sim}{\longrightarrow} 
\pi_1^{ab}(\X_K)_W^{geo}\hookrightarrow \pi_1^{ab}(\X_K)^{geo}$$
where 
\begin{itemize}
\item $H^{3}_{ar}(\X_{K},\Z(2))^0$ is finitely generated.
\item $s^0$ is surjective and $\ker s^0$ the maximal divisible subgroup of 
$SK_1(\X_K)^0$.
\item $i^0$ is the inclusion of the torsion subgroup 
of $H^{3}_{ar}(\X_{K},\Z(2))^0$.
\item The last map is the inclusion of a finitely generated group into its 
profinite completion.
\end{itemize}
\end{thm}

\begin{proof}
This follows by combining Theorems \ref{neutheorem}, \ref{thmRec1}
and Proposition \ref{proplast}.
\end{proof}

\end{document}